\documentclass[12pt]{amsart}
\usepackage[english]{babel}
\usepackage[letterpaper,top=2cm,bottom=2cm,left=3cm,right=3cm,marginparwidth=1.75cm]{geometry}
\usepackage{amsmath}
\usepackage{amssymb}
\usepackage{amsthm}
\usepackage{graphicx}
\usepackage{tikz}
\usetikzlibrary{positioning, arrows.meta}
\usepackage[colorlinks=true, allcolors=blue]{hyperref}
\usepackage{lmodern}
\newcommand{\supertiny}{\fontsize{5pt}{12pt}\selectfont} 

\theoremstyle{definition}
\newtheorem{theorem}{Theorem}[section]
\newtheorem{lemma}[theorem]{Lemma}
\newtheorem{proposition}[theorem]{Proposition}
\newtheorem{corollary}[theorem]{Corollary}
\newtheorem{definition}[theorem]{Definition}
\newtheorem{remark}[theorem]{Remark}
\newtheorem{example}[theorem]{Example}
\newtheorem{question}[theorem]{Question}

\DeclareMathOperator{\im}{im}
\DeclareMathOperator{\st}{st}
\DeclareMathOperator{\lk}{lk}

\title{Algebraic fiberings of right-angled Coxeter groups on bipartite cone graphs}
\author{Kasia Jankiewicz}
\author{Shayan Miri}
\address{University of British Columbia, Vancouver, BC, Canada V6T 1Z2}
\email{kasia@math.ubc.ca}
\email{smiri01@student.ubc.ca}

\begin{document}
\maketitle

\begin{abstract}
We study systems of moves and states in the sense of \cite{JankiewiczNorinWise16} on certain bipartite graphs, and the associated homomorphisms from the commutator subgroups of right-angled Coxeter groups to $\mathbb Z$. We give examples of right-angled Coxeter groups whose graphs do not satisfy the legal system criterion, but contain finite index right-angled Coxeter subgroups whose graphs admit legal systems. We also provide a family of examples of graphs without legal systems, where no sequence of the considered graph operations yields a graph with a legal system. Moreover, we show that for a bipartite graph and an orbit of non-empty states under the bipartite system of moves, the associated homomorphism from the commutator subgroup of the right-angled Coxeter group to $\mathbb Z$ has finitely generated kernel if and only if the system is legal.
\end{abstract}


\section{Introduction}

Bestvina--Brady Morse theory \cite{BestvinaBrady97} allows analyzing finiteness conditions of homomorphisms to $\mathbb Z$ given by Morse functions on affine complexes. It was applied in \cite{JankiewiczNorinWise16} to give a combinatorial criterion for virtual algebraic fibering of right-angled Coxeter groups, i.e.\ homomorphisms from their finite index subgroups to $\mathbb Z$ with finitely generated kernels. The criterion requires a \emph{legal system} on the defining graph $\Gamma$ of a right-angled Coxeter group $W_{\Gamma}$. See Section~\ref{sec: legal system} for a precise definition. A legal system on the defining graph ensures the existence of an algebraic fibering on the commutator subgroup of $W_{\Gamma}$. The criterion and its variants have been used in the construction of fibrations of various groups e.g.\ in \cite{ItalianoMartelliMigliorni215manifold, ItalianoMartelliMigliorni21dim8,  ScheslerZaremsky2023, Lafontetal2025}.
However, the criterion does not always detect whether a right-angled Coxeter group virtually algebraically fibers, as was demonstrated in \cite{JankiewiczNorinWise16}.

In this paper we propose expanding the search for algebraic fiberings to other finite index subgroups of $W_{\Gamma}$, which can still be analyzed combinatorially, and we provide examples where the strategy is successful.

Specifically, we focus on a class of bipartite graphs,  referred to as \emph{bipartite cones}, that were considered in \cite{JankiewiczNorinWise16} and were shown to not admit legal systems for appropriate values of $n,k$.
\begin{definition}[Bipartite cone]
For integers $1 \leq k \leq n$, the \emph{bipartite cone graph} $\Lambda(n,k)$
is the graph with the vertex set consisting of two types of vertices:
$$
1, 2, \ldots, n \qquad \text{(\emph{the number vertices})},$$
$$
 \sigma \subseteq \{1,\ldots,n\} \ :\ |\sigma| = k  \qquad \text{(\emph{the $k$-subset vertices})},$$
and the edge set: 
$$\bigl\{\, (i, \sigma) \ :\ i \text{ is a number},\ \sigma \text{ is a $k$-subset, and}\ i \in \sigma \,\bigr\}.$$
\end{definition}

 We consider three operations on graphs, which yield new right-angled Coxeter groups which naturally embed as finite index subgroups of the original group: vertex link doubling, simultaneous vertex link multiplying, and edge quadrupling. See Section~\ref{sec: graph operations} for precise definitions. We obtain the following classification of which bipartite cones admit legal systems, after possibly applying the graph operations.

\begin{theorem}\label{thm: main}
    Let $\Lambda = \Lambda(n,k)$ be a bipartite cone. Let $k\geq 3$, and $n>k$.
    \begin{enumerate}
        \item\label{item1} If $n<2k-1$, then $\Lambda$ has a legal system. Consequently, $W_\Lambda$ virtually algebraically fibers.
        \item\label{item2} If $n=2k-1$, then $\Lambda$ does not have a legal system, but there exists a graph $\Gamma$ with a legal system, whose associated right-angled Coxeter group embeds as a finite index subgroup of $W_{\Lambda}$. Consequently, $W_\Lambda$ virtually algebraically fibers.
        \item\label{item3} If $n>2k-1$, then $\Lambda$ does not have a legal system, and no finite sequence of considered graph operations yields a graph $\Gamma$ with a legal system. Moreover for any such $\Gamma$ and any orbit of non-empty states under the bipartite system of moves on $\Gamma$, the induced homomorphism $W_{\Gamma}'\to \mathbb Z$ has infinitely generated kernel.
    \end{enumerate}
\end{theorem}

The infinite generation of the kernels $\ker(W_{\Gamma}'\to \mathbb Z)$ in Theorem~\ref{thm: main}\eqref{item3} follows from a more general statement which holds for every bipartite graph, and can be of independent interest. 

\begin{theorem}\label{thm: main2}
    Let $\Gamma$ be a bipartite graph with the vertex set partition $V(\Gamma) = L\sqcup R$. Let $\mathcal M = \{L, R\}$ be the bipartite system of moves, and let $\mathcal S$ be an orbit of non-empty states under $\mathcal M$. Let $\chi: W_{\Gamma}'\to \mathbb Z$ be the associated homomorphism. Then $\ker \chi$ is finitely generated if and only if $\mathcal S$ is a legal orbit. 
\end{theorem}

\subsection*{Organization}
In Section~\ref{sec: preliminaries} we discuss right-angled Coxeter groups and the combinatorial game on the defining graph which ensures the existence of virtual algebraic fiberings of right-angled Coxeter groups. We also review some graph operations that yield commensurable right-angled Coxeter groups, and show that these operations preserve bipartiteness. Section~\ref{sec: legal systems} contains the proofs of Theorem~\ref{thm: main}\eqref{item1} and Theorem~\ref{thm: main}\eqref{item2}. Section~\ref{sec: obstructions} focuses on obstructions to the existence of legal systems, and we show that they are present in every graph obtained from $\Lambda(n,k)$ with $n>2k-1$ by a sequence of the graph operations defined in Section 2, proving Theorem~\ref{thm: main}\eqref{item3}.  
In Section~\ref{sec: infinite generations} we prove Theorem~\ref{thm: main2}.

\subsection*{Acknowledgements} The first author was partially supported by the Natural Sciences and Engineering Research Council of Canada (NSERC). The second author was partially supported by the University of British Columbia (UBC) Work Learn International Undergraduate Research Award (WLIURA). The first author thanks Kevin Schreve, Grigori Avramidi, and Annette Karrer for useful conversations.

Some AI tools were used to prove results presented in Section~\ref{sec: infinite generations}. In other parts, they were only used for making figures and light edits to the text. The article was entirely written by the authors, who are fully responsible for its correctness.

\section{Preliminaries}\label{sec: preliminaries}
Throughout this paper $\Gamma = (V,E)$ is always a simplicial graph.
\subsection{Right-angled Coxeter groups}

The \emph{right-angled Coxeter group} associated to $\Gamma$ is
$$W_\Gamma = \bigl\langle\, v \in V(\Gamma) \ \big|\ v^2 = 1,\ \ (vw)^2 = 1 \text{ whenever } (v,w) \in E(\Gamma) \,\bigr\rangle.$$

\subsubsection{The complex $X_{\Gamma}$}\label{sec: X_Gamma}
Let $W_\Gamma' = [W_\Gamma, W_\Gamma]$ be the commutator subgroup, i.e.\ $W_{\Gamma'}$ is the kernel of the abelianization homomorphism $W_{\Gamma}\to \mathbb Z/2\mathbb Z^{|V(\Gamma)|}$. Let $X_{\Gamma} = \Sigma_{\Gamma}/W_{\Gamma}'$ be the quotient of the Davis complex $\Sigma_{\Gamma}$ of $W_{\Gamma}$ by the action of $W_{\Gamma}'$ \cite{DavisCoxeterBook2008}. 

More explicitly, the complex $X_{\Gamma}$ is a non-positively curved cube complex whose $0$-cells are in 1-to-1 correspondence with $\{0,1\}^{V(\Gamma)}$, the $1$-cells correspond to the pairs of vertices differing on a single coordinate, and each edge can be naturally labeled by a vertex of $\Gamma$. In particular, the $1$-skeleton of $X_{\Gamma}$ is the $1$-skeleton of the $|V(\Gamma)|$-cube.
Each family of parallel $1$-cells of $X_{\Gamma}$ is labeled by $u \in V(\Gamma)$, and the $2$-cells of $X_{\Gamma}$ are spanned by two pairs of parallel $1$-cells whose labels are adjacent in $\Gamma$. The higher dimensional cubes are defined analogously and correspond to bigger cliques in $\Gamma$. In this text, we consider bipartite graphs, and so their associated complexes are $2$-dimensional.

\subsection{Legal systems and states}\label{sec: legal system}

A \emph{diagonal map} on a cube complex $X$ is a map $f: X\to S^1$ which restricted to each $1$-cell is a degree $1$ map to $S^1$, and restricted to each $k$-cube it factors through a projection to one of the big diagonals of the cube. 
A diagonal map $f: X\to S^1$ lifts to a Morse function $\widetilde f: \widetilde X \to \widetilde S^1 = \mathbb R$, in the sense of Bestvina-Brady \cite{BestvinaBrady97}.

\subsubsection{System of moves and states}
\label{sec: states and moves}

Let $\Gamma = \Gamma(V, E)$ be a simplicial graph. A \emph{state} of $\Gamma$ is a subset $S \subseteq V$. The following notions were defined in \cite{JankiewiczNorinWise16}. A state $S$ is \emph{legal} if the subgraphs induced by $S$ and $S^c :=V - S$ are both nonempty and connected. 

A \emph{move} at $v \in V$ is a subset $M_v \subseteq V$ such that $v \in M_v$ and $u \notin M_v$ for all $\{u,v\} \in E$. A \emph{system of moves} is a choice of move $M_v$ for each $v \in V$. 
The system of moves $\mathcal M = \{M_v:v\in V(\Gamma)\}$ \emph{acts} on the set of all states of $\Gamma$ by a symmetric difference:
$$M_v\cdot S := S\Delta M_v.$$
This action factors through an action of $(\mathbb Z/2)^{|\mathcal M|}$. 
The system is \emph{legal} if there exists an orbit of states under the action of the moves all of whose elements are legal states.

A system of moves $\{M_v\subseteq V(\Gamma) : v\in V(\Gamma)\}$ and a single state $S_0\subseteq V(\Gamma)$ determine a diagonal map $f:X_{\Gamma}\to S^1$ as follows. 
Consider a cellular structure on $S^1$ with a single $0$-cell $o$ and one oriented $1$-cell $e$. Fix $x_0\in X_\Gamma^0$. Then each $x\in X^0$ can be obtained by traveling from $x_0$ by traversing one $1$-cell of each label in some set of labels $\{u_1, \dots, u_k\}\subseteq V(\Gamma)$ (independent of the order of traversing edges). 
The map $f:X\to S^1$ maps each $0$-cell to $o$, and maps a $1$-cell starting at $x$ with label $u$ onto $e$ with positive orientation if $u\in S(x):=\prod_{i=1}^k M_{u_i}\cdot S_0$, and onto $e$ with negative orientation otherwise. Note that the order of $M_{v_i}$ in the above expression does not play a role, and therefore $f$ is well-defined.
The following lemma guarantees that the induced homomorphism $\chi  = f_*:W_{\Gamma}'\to \mathbb Z$ is nontrivial.
\begin{lemma}\label{lem: nontrivial homomorphism}
    If for every $x\in X_\Gamma^0$, $S(x)$ is a proper subset of $V(\Gamma)$, then $\chi:W_{\Gamma}'\to \mathbb Z$ is non-trivial.
\end{lemma}

\begin{proof}
The map $f\colon X_\Gamma\to S^1$ lifts to $\widetilde f\colon \widetilde X_\Gamma=\Sigma_\Gamma\to\mathbb R$,
which is $\chi$-equivariant, i.e.\ $\widetilde f(g\cdot y)=\widetilde f(y)+\chi(g)$ for all
$g\in W_\Gamma'$ and $y\in\Sigma_\Gamma$. If $\chi$ were trivial, $\widetilde f$ would be
$W_\Gamma'$-invariant, and hence bounded, since $W_\Gamma'$ acts cocompactly on $\Sigma_\Gamma$.
So it suffices to show that $\widetilde f$ is unbounded.

Let $\widetilde x$ be any $0$-cell of $\Sigma_\Gamma$, with image $x\in X_\Gamma$. Since $S(x)$ is a
proper subset of $V(\Gamma)$, there is $w\notin S(x)$, and the edge at $\widetilde x$ labeled by
$w$ is mapped by $\widetilde f$ onto $[\widetilde f(\widetilde x)-1,\widetilde f(\widetilde x)]$.
Its other endpoint $\widetilde x'$ thus satisfies $\widetilde f(\widetilde x')=\widetilde f(\widetilde x)-1$.
Applying the same argument to $\widetilde x'$ and iterating, we obtain $0$-cells on which
$\widetilde f$ takes the values $\widetilde f(\widetilde x)-k$ for all $k\geq 0$. Hence
$\widetilde f$ is unbounded.
\end{proof}

An \emph{algebraic fibering} of a group $G$ is a non-trivial homomorphism $\chi: G\to \mathbb Z$ with finitely generated $\ker \chi$. The following theorem is a consequence of the Bestvina-Brady Morse theory \cite{BestvinaBrady97}.
\begin{theorem}[{\cite[Thm 4.3]{JankiewiczNorinWise16}}]\label{thm: JNW}
    Suppose that $\Gamma$ admits a legal system. Then the associated map $f:X_{\Gamma}\to S^1$ induces an algebraic fibering of $W_{\Gamma}'$. In particular, $W_{\Gamma}$ virtually algebraically fibers. 
\end{theorem}

\subsection{Bipartite moves}
In this paper we will be interested in particularly simple systems of moves.

\begin{definition}
    Let $\Gamma$ be a bipartite graph with the partition of the vertex set $V(\Gamma) = L \sqcup R$. The \emph{bipartite system of moves} is $\mathcal M = \{L,R\}$ where 
    $$M_v = \begin{cases}
        L \text{ when } v\in L\\
        R \text{ when } v\in R\\
    \end{cases}$$
\end{definition}
\begin{lemma}\label{lem:bipartite moves states}
Let $\Gamma$ be a bipartite graph with the partition of the vertex set $V(\Gamma) = L \sqcup R$, and let $\mathcal M$ be the bipartite system of moves. An orbit of states under the action of $\mathcal M$ corresponds to partitions $L = L_0\sqcup L_1$ and $R = R_0\sqcup R_1$, where the states are $S_{ij} = L_i\cup R_j$ for $i,j\in \{0,1\}$. Such a system is legal if and only if each state $S_{ij}$ is non-empty and spans a connected subgraph of $\Gamma$.
\end{lemma}
\begin{proof}
    Given an initial state $S\subseteq V(\Gamma)$, let $L_0 = S\cap L$ and $R_0 = S\cap R$. Then the orbit of $S$ under the action of $\mathcal M$ is exactly $\{S_{00}, S_{01}, S_{10}, S_{11}\}$, and $\mathcal M$ acts on them by swapping $L_0 \leftrightarrow L_1$ (the move $L$) and $R_0 \leftrightarrow R_1$ (the move $R$). 
    Since $V - S_{ij} = S_{(1-i)(1-j)}$ for $i,j \in \{0,1\}$, these four states pair up into two complementary pairs, and checking legality of each state is equivalent to checking non-emptiness and connectedness of each graph spanned by $S_{ij}$.
\end{proof}

\begin{remark}(The $4$-vertex quotient $\overline X_{\Gamma}$ of $X_{\Gamma}$)\label{rem: 4vertex complex}
    Given a bipartite system of moves $\mathcal M = \{L,R\}$ and an orbit of states $\{S_{ij}\}_{i,j\in\{0,1\}}$ as above, the diagonal map $X_{\Gamma}\to S^1$ factors through the complex $\overline X_{\Gamma}$ that is the quotient of the Davis complex $\Sigma_\Gamma$ by the action of the (torsion-free) kernel $K_{\Gamma}$ of the homomorphism $W_{\Gamma}\to \mathbb Z/2\mathbb Z\times \mathbb Z/2\mathbb Z$ which sends each generator from $L$ to $(1,0)$ and each generator from $R$ to $(0,1)$.  In other words, the complex $\overline X_{\Gamma}$ has four vertices which we denote $x_{00}, x_{10}, x_{01}, x_{11}$, the following edges: 
    \begin{itemize}
        \item for $j=0,1$ one edge $e_{\ell, j}$ for each $\ell\in L$ between $x_{0j}$ and $x_{1j}$,
        \item for $i=0,1$ one edge  $e_{i, r}$ for each $r\in R$ between $x_{i0}$ and $x_{i1}$,
    \end{itemize}
     and a square $s_{\ell r}$ with the four distinct vertices, edges $e_{\ell,0}, e_{\ell,1}, e_{0,r}, e_{1,r}$ for each $(\ell, r)\in E(\Gamma)$.
    In particular, the induced homomorphism $\chi: W_{\Gamma}'\to \mathbb Z$ extends to $\overline\chi: \pi_1(\overline X_{\Gamma}) =K_{\Gamma}\to \mathbb Z$.
\end{remark}

\subsection{Graph operations and subgroups of right-angled Coxeter groups}\label{sec: graph operations}
For $v\in V(\Gamma)$, a \emph{neighbor} of $v$ is any $u\in V(\Gamma)$ such that $(v,u)\in E$. We denote the set of all neighbors of $v$ by $N(v)$. The \emph{link} $\lk(v)$ of $v$ is the subgraph of $\Gamma$ spanned by $N(v)$. The \emph{star} $\st(v)$ of $v$ is the subgraph spanned by $N(v) \cup \{v\}$.

Below we describe a few constructions of graphs obtained from a graph $\Gamma$ that yield defining graphs of right-angled Coxeter groups commensurable with $W_{\Gamma}$. All these constructions are known to the experts, see e.g. \cite{PallaviStarkThomas18, BoundxXie20, CashenDaniEdletzbergerKarrer}.

All three constructions below have the same shape. We delete a set $X$ of vertices, take several copies $\Gamma^\alpha$
of $\Gamma - X$ (or of $\Gamma-\{v_i\}$), indexed by a finite set $I$, and glue them along the
vertices of $\Gamma$ that are adjacent to the corresponding elements of $X$.
Concretely, for each $\alpha\in I$ we specify a set $D_\alpha\subseteq V(\Gamma)$ and an injective map
$\iota_\alpha\colon D_\alpha\to V(\Gamma')$, and then
$$E(\Gamma') = \bigl\{(\iota_\alpha(w),\iota_\alpha(w')) : \alpha\in I,\ w,w'\in D_\alpha,\ (w,w')\in E(\Gamma)\bigr\}.$$
We write $\Gamma^\alpha$ for the subgraph of $\Gamma'$ spanned by $\iota_\alpha(D_\alpha)$; it is isomorphic to the subgraph of $\Gamma$ spanned by $D_\alpha$.

We start with a vertex link double, which is obtained from taking two copies of $\Gamma-\{v\}$, and identifying them along $\lk(v)$.
\begin{definition}[Vertex link double]\label{defn: link double} Let $v\in V(\Gamma)$ and $B := V(\Gamma)-\st(v)$. The \emph{vertex link double} of $\Gamma$
along $v$ is the graph $\Gamma'$ with
$$V(\Gamma') = B^0\sqcup B^1\sqcup N(v),$$
where $B^i=\{b^i : b\in B\}$ are disjoint copies of $B$, and
$$E(\Gamma') = \bigl\{(\iota_i(w),\iota_i(w')) : i=0,1,\ w,w'\in V(\Gamma) -\{v\},\ (w,w')\in E(\Gamma)\bigr\}$$
where for $i=0,1$: $\iota_{i}(w) = \begin{cases}
    w^i \text{ if }w\in B\\ w \text{ if }w\in N(v)
\end{cases}$.

Thus $\Gamma'=\Gamma^0\cup_{\lk(v)}\Gamma^1$ is the union of two copies $\Gamma^0, \Gamma^1$  of $\Gamma-\{v\}$ glued
along $\lk(v)$. See Figure~\ref{fig:linkdouble-generic}.
\begin{figure}
\centering
\begin{tikzpicture}[
    v/.style={draw=black, fill=blue!15, circle, minimum size=1.2cm, font=\small},
    >=Stealth, thick
]
\node[v, fill=orange!15] (B0) at (0,0) {$B^0$};
\node[v, fill=gray!15] (Nv) at (4,0) {$N(v)$};
\node[v, fill=orange!12] (B1) at (8,0) {$B^1$};
\draw[-] (B0) -- (Nv) node[midway, above, font=\tiny] {
edges as in $\Gamma$
};
\draw[-] (Nv) -- (B1) node[midway, above, font=\tiny] {
edges as in $\Gamma$
};
\end{tikzpicture}
\caption{The link double of $\Gamma$
along $v$ consists of two disjoint copies $B^0, B^1$ of $B = \Gamma - v - N(v)$, each
connected to the single shared copy of $N(v)$ exactly as $B$ was connected
to $N(v)$ in $\Gamma$. In particular, there are no edges between $B^0$ and $B^1$.}
\label{fig:linkdouble-generic}
\end{figure}
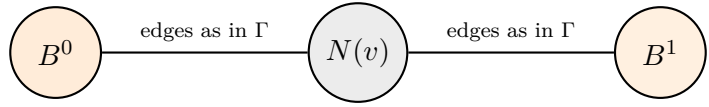
\end{definition}

More generally given a collection of vertices with disjoint stars, one can apply the above construction simultaneously. We call the result of this operation simultaneous vertex link multiple.

\begin{definition}[Simultaneous vertex link multiple]\label{defn: link multiple}
Let $v_1,\dots,v_m\in V(\Gamma)$ have pairwise disjoint stars.
The \emph{simultaneous vertex link multiple} of $\Gamma$ along $v_1,\dots,v_m$ is the graph
$\Gamma'$ with the vertex set
$$V(\Gamma') = B^0\sqcup\bigsqcup_{i=1}^m\bigl(N(v_i)\sqcup B^i\bigr),$$
where 
$B^0 = \{b: b\in V(\Gamma)- \bigcup_{i=1}^m \st(v_i)\}$ and  $B^i = \{b^i : b\in V(\Gamma)-\st(v_i)\}$ for $i=1, \dots, m$,
and the edge set 
$$E(\Gamma') = \bigl\{(\iota_i(w),\iota_i(w')) : i=0,\dots, m,\ w,w'\in D_i,\ (w,w')\in E(\Gamma)\bigr\}$$
where 
$D_0 = V(\Gamma)-\{v_1,\dots,v_m\}$, $D_i = V(\Gamma)-\{v_i\}$ for $i\geq 1$, and
$$\iota_0(w)=w \text{ for all }w\in D_0,$$
$$\iota_i(w)=\begin{cases}
    w^i &\text{ if }w\in V(\Gamma)-\st(v_i)\\w &\text{ if } w\in N(v_i)\end{cases}$$
    for $i=1,\dots,m$.

Thus $\Gamma'$ is the union of $\Gamma^0\cong\Gamma-\{v_1,\dots,v_m\}$ and
$\Gamma^i\cong\Gamma-\{v_i\}$, where $\Gamma^i$ meets $\Gamma^0$ in $N(v_i)$ and
$\Gamma^i\cap\Gamma^j=\emptyset$ for $1\leq i<j$. 
See Figure~\ref{fig:simultaneous link multiple}.
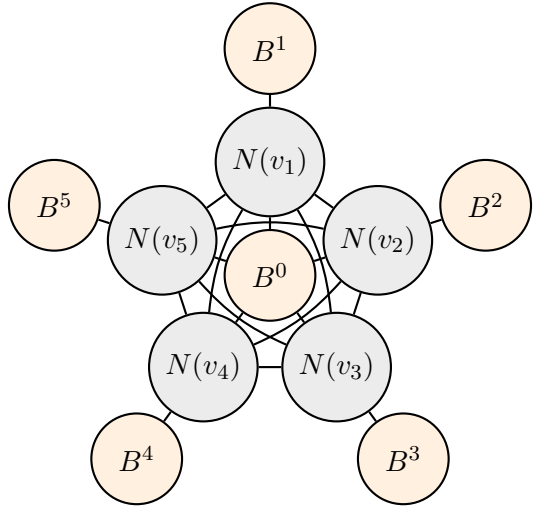
\begin{figure}
\centering
\begin{tikzpicture}[
    v/.style={draw=black, fill=blue!15, circle, minimum size=1.2cm, font=\small},
    >=Stealth, thick
]
\node[v, fill=orange!12] (B0) at (0,0) {$B^0$};
\node[v, fill=gray!15] (Nv3) at (306:1.5) {$N(v_3)$};
\node[v, fill=gray!15] (Nv2) at (18:1.5) {$N(v_2)$};
\node[v, fill=gray!15] (Nv1) at (90:1.5) {$N(v_1)$};
\node[v, fill=gray!15] (Nv4) at (234:1.5) {$N(v_4)$};
\node[v, fill=gray!15] (Nv5) at (162:1.5) {$N(v_5)$};
\node[v, fill=orange!12] (B3) at (306:3) {$B^3$};
\node[v, fill=orange!12] (B2) at (18:3) {$B^2$};
\node[v, fill=orange!12] (B1) at (90:3) {$B^1$};
\node[v, fill=orange!12] (B4) at (234:3) {$B^4$};
\node[v, fill=orange!12] (B5) at (162:3) {$B^5$};

\draw[-] (B0) -- (Nv1) -- (B1);
\draw[-] (B0) -- (Nv2) -- (B2);
\draw[-] (B0) -- (Nv3) -- (B3);
\draw[-] (B0) -- (Nv4) -- (B4);
\draw[-] (B0) -- (Nv5) -- (B5);
\draw[-] (Nv1) -- (Nv2) -- (Nv3) -- (Nv4) -- (Nv5) -- (Nv1);
\foreach \a/\b/\t in {1/3/90,2/4/18,3/5/306,4/1/234,5/2/162}
  \draw (Nv\a) to[out=\t-150, in=\t+6] (Nv\b);
\end{tikzpicture}
\caption{The simultaneous vertex link multiple of $\Gamma$ along $v_1, v_2, v_3, v_4, v_5$.}
\label{fig:simultaneous link multiple}
\end{figure}
\end{definition}

The last construction we consider is the edge quadruple of $\Gamma$ along an edge $(u,v)$. 

\begin{definition}[Edge quadruple]\label{defn: edge quadruple}
Let $(u,v)\in E(\Gamma)$, and write
$$B := V(\Gamma) - \bigl(\st(u) \cup \st(v)\bigr),\;\;
N_u := N(u) - \st(v),\;\;
N_v := N(v) - \st(u),\;\; 
N_{uv} := N(u) \cap N(v)$$
so that $V(\Gamma) - \{u,v\} = B\sqcup N_u\sqcup N_v\sqcup N_{uv}$. The \emph{edge quadruple}
of $\Gamma$ along $(u,v)$ is the graph $\Gamma'$ with the vertex set
$$V(\Gamma') = \Big(\bigsqcup_{i,j\in\{0,1\}}B^{ij}\Big)\sqcup\Big(\bigsqcup_{j\in\{0,1\}}N_u^j\Big)
\sqcup\Big(\bigsqcup_{i\in\{0,1\}}N_v^i\Big)\sqcup N_{uv},$$
where $B^{ij}=\{b^{ij}:b\in B\}$, $N_u^j=\{x^j:x\in N_u\}$, $N_v^i=\{y^i:y\in N_v\}$ are
disjoint copies, and the edge set 
$$E(\Gamma') = \bigl\{(\iota_{ij}(w),\iota_{ij}(w')) : i,j=0,1,\ w,w'\in V(\Gamma)-\{u,v\},\ (w,w')\in E(\Gamma)\bigr\}$$
where 
$$\iota_{ij}(w)=
\begin{cases}
    w^{ij} &\text{ if }w\in B\\ w^j &\text{ if }w\in N_u\\ w^i &\text{ if }w\in N_v\\ w &\text{ if }w\in N_{uv}\\
\end{cases}.$$
Thus $\Gamma'$ is the union of four copies $\Gamma^{ij}$ of $\Gamma-\{u,v\}$, where
$\Gamma^{ij}$ and $\Gamma^{i'j'}$ share $N_u^j$ if $j=j'$, share $N_v^i$ if $i=i'$, and always
share $N_{uv}$. 
See Figure~\ref{fig:edge quadruple}.

\begin{figure}
\centering
\begin{tikzpicture}[
    v/.style={draw=black, fill=blue!15, circle, minimum size=1.2cm, font=\small},
    >=Stealth, thick
]
\node[v, fill=orange!12] (Nab) at (0,0) {$N_{uv}$};
\node[v, fill=gray!15] (Na1) at (330:2) {$N_u^1$};
\node[v, fill=purple!15] (Nb0) at (30:2) {$N_v^0$};
\node[v, fill=gray!15] (Na0) at (150:2) {$N_u^0$};
\node[v, fill=purple!15] (Nb1) at (210:2) {$N_v^1$};
\node[v, fill=blue!15] (B01) at (0: 3.25) {$B^{01}$};
\node[v, fill=blue!15] (B00) at (90: 2.25) {$B^{00}$};
\node[v, fill=blue!15] (B10) at (180: 3.25) {$B^{10}$};
\node[v, fill=blue!15] (B11) at (270 : 2.25) {$B^{11}$};

\draw[-] (Nab) -- (Na0) -- (Nb0) -- (Nab) -- (Nb1) -- (Na1) -- (Nab);
\draw[-] (Nb1) -- (Na0);
\draw[-] (Na1) -- (Nb0);
\draw[-] (B00) -- (Na0) -- (B10) -- (Nb1) -- (B11) -- (Na1) -- (B01) -- (Nb0) -- (B00);
\end{tikzpicture}
\caption{The edge quadruple of $\Gamma$ along $e = (u,v)$.}
\label{fig:edge quadruple}
\end{figure}
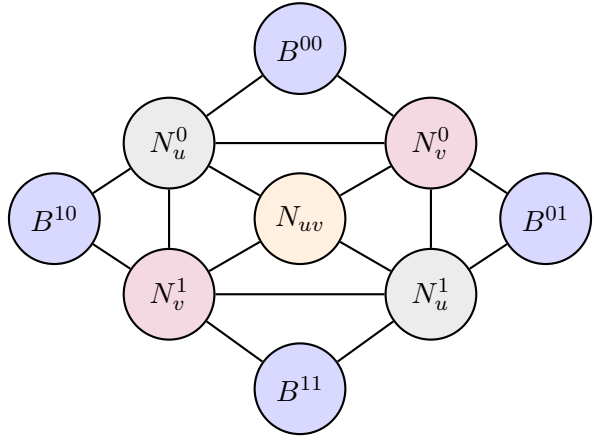
\end{definition}

The above constructions are of interest to us, because the associated right-angled Coxeter groups on graphs obtained from $\Gamma$ via these constructions have finite index in $W_{\Gamma}$. 
\begin{proposition} Let $\Gamma'$ be a graph obtained from $\Gamma$ by vertex link doubling, edge quadrupling, or simultaneous vertex link multiplying. Then the right-angled Coxeter group $W_{\Gamma'}$ embeds as a finite index subgroup in $W_{\Gamma}$.
\end{proposition}

\begin{proof}
    For the link double along $v$, we claim that $W_{\Gamma'}$ is isomorphic to the kernel of the retraction $r: W_{\Gamma}\to \langle v\rangle \simeq\mathbb Z/2$ onto $v$. Indeed, the generators of $\ker r$ are all of the form $u, u^v$ for $u\in V(\Gamma)-\{v\}$. Since $u = u^v$ when $u\in N(v)$, the complete list of generators are 
    $u\in N(v)$ and $b, b^v$ for $b\in B$, which we denote by $b^0, b^1\in V(\Gamma')$ respectively. By \cite{Deodhar1989, Dyer1990} $\ker r$ is indeed a right-angled Coxeter group on these generators, and with the commuting relation encoded by the graph $\Gamma'$.
    
    For the edge quadruple along $(u,v)$, we claim that $W_{\Gamma'}$ is isomorphic to the kernel of the retraction $r:W_{\Gamma}\to \langle u,v\rangle \simeq \mathbb Z/2\times \mathbb Z/2$. The generators of $\ker r$ are of the form $w, w^u, w^v, w^{uv}$. Since $w = w^u = w^v = w^{uv}$ when $w\in N_{uv}$ there is a unique generator from $N_{uv}$ in $W_{\Gamma'}$. For $w\in N_u$ we have $w = w^u$ and $w^v = w^{uv}$, we denote these by $w^0, w^1\in V(\Gamma')$. Similarly, for $w\in N_v$ we have $w = w^v$ and $w^u =w^{uv}$, we denote these by $w^0, w^1\in V(\Gamma')$. Finally, for $w\in B$, $w, w^u, w^v, w^{uv}$ are all distinct, and denoted by $w^{00}, w^{10}, w^{01}, w^{11} \in V(\Gamma')$. Again, by \cite{Deodhar1989, Dyer1990} $\ker r$ is a right-angled Coxeter group on all these generators listed above, and so the commuting relations are encoded by $\Gamma'$.
    A detailed proof of this proposition for the vertex link double and edge quadruple, as well as a more general procedure of $2^n$-multiple along an $n$-simplex, can be also found in \cite[Prop 5.2]{BoundxXie20}.

    We now explain the case of the simultaneous vertex link multiple (which is also known to the experts). Since $v_1,\dots, v_m$ have disjoint stars, the subgroup $H = \langle v_1, \dots, v_m\rangle_{W_{\Gamma}}$ of $W_{\Gamma}$ is isomorphic to the free product $\mathbb Z/2* \dots * \mathbb Z/2$ of $m$ copies of $\mathbb Z/2$. Let $\phi: W_{\Gamma}\to H$ be the retraction onto $H$. 
    Let $K =\langle v_iv_jv_i :i\neq j\rangle_H$ of $H$, which by Lemma~\ref{lem: free product of Z/2} has index $[H:K]=m+1$. We claim that the group $W_{\Gamma'}$ is the preimage $\phi^{-1}(K)$. Since ${e, v_1, \dots, v_m}$ is a transversal of $K$ in $H$, the generators of $\phi^{-1}(K)$ are the generators of $V(\Gamma)-\{v_1, \dots, v_m\}$ together with all the conjugates $u^{v_i}$ for $i=1, \dots, m$, and $u\in V(\Gamma)-\{v_i\}$. We note that $u = u^{v_i}$ for all $u\in N(v_i)$, and $u^{v_i}\neq u^{v_j}$ for all $j\neq i$, since by assumption $N(v_i)\cap N(v_j) = \emptyset$.
    Thus all the generators of $\phi^{-1}(K)$ can be organized as in the definition~\ref{defn: link multiple} where $u^i$ denotes $u^{v_i}$.  
    As before, by \cite{Deodhar1989, Dyer1990} $\phi^{-1}(K)$ is a right-angled Coxeter group on these generators, and so the commuting relations are encoded by $\Gamma'$.
\end{proof}

\begin{lemma}\label{lem: free product of Z/2}
    Let $H = \langle v_1, \dots, v_m\mid v_i^2=1 \text{ for all }i=1,\dots, m\rangle$. Let $K$ be the subgroup of $H$ generated by all the conjugates $v_i^{v_j} = vjv_iv_j$ of the generators by the generators with $i\neq j$. Then the index $[H:K] = m+1$.
\end{lemma}

\begin{proof}
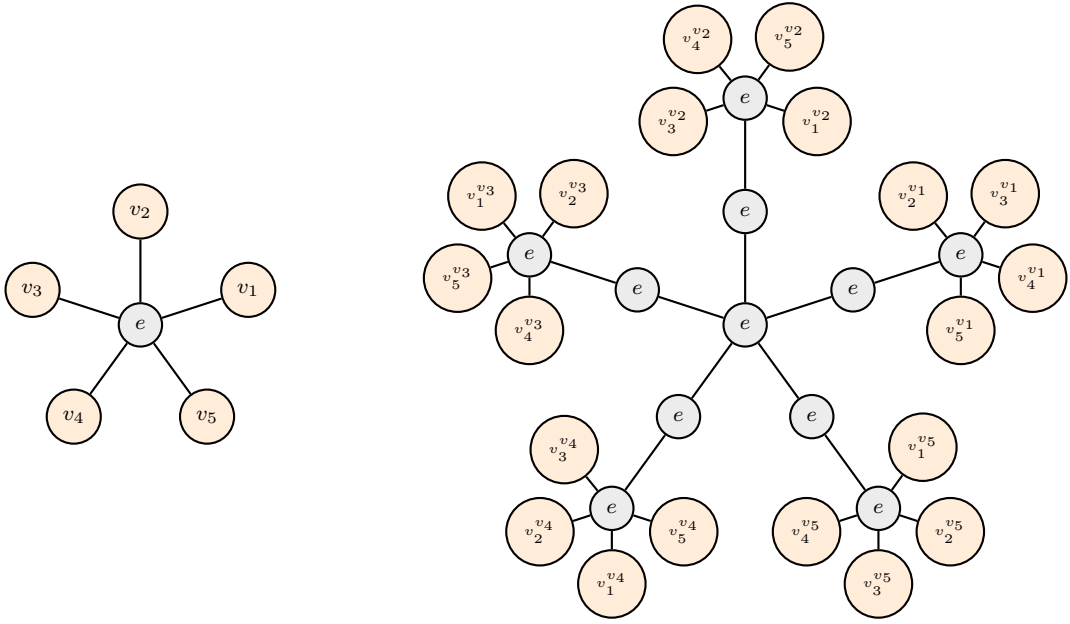
\begin{figure}
\centering
\begin{tikzpicture}[
    v/.style={draw=black, fill=blue!15, circle, minimum size=0.1cm, font=\small},
    >=Stealth, thick
]
\begin{scope}[shift={(-8,0)}]
\node[v, fill=gray!15] (G0) at (0,0) {\tiny $e$};
\node[v, fill=orange!15] (G1) at (18:1.5) {\tiny$v_1$};
\node[v, fill=orange!15] (G2) at (90:1.5) {\tiny$v_2$};
\node[v, fill=orange!15] (G3) at (162:1.5) {\tiny$v_3$};
\node[v, fill=orange!15] (G4) at (234:1.5) {\tiny$v_4$};
\node[v, fill=orange!15] (G5) at (306:1.5) {\tiny$v_5$};
\draw (G0) -- (G1);
\draw (G0) -- (G2);
\draw (G0) -- (G3);
\draw (G0) -- (G4);
\draw (G0) -- (G5);
\end{scope}

\begin{scope}
\node[v, fill=gray!15] (G0) at (0,0) {\tiny $e$};
\node[v, fill=gray!15] (G1') at (18:1.5) {\tiny $e$};
\node[v, fill=gray!15] (G2') at (90:1.5) {\tiny $e$};
\node[v, fill=gray!15] (G3') at (162:1.5) {\tiny $e$};
\node[v, fill=gray!15] (G4') at (234:1.5) {\tiny $e$};
\node[v, fill=gray!15] (G5') at (306:1.5) {\tiny $e$};

\node[v, fill=gray!15] (G1) at (18:3) {\tiny $e$};
\node[v, fill=gray!15] (G2) at (90:3) {\tiny $e$};
\node[v, fill=gray!15] (G3) at (162:3) {\tiny $e$};
\node[v, fill=gray!15] (G4) at (234:3) {\tiny $e$};
\node[v, fill=gray!15] (G5) at (306:3) {\tiny $e$};

\draw (G0) --  (G1') -- (G1);
\draw (G0) --  (G2') -- (G2);
\draw (G0) --  (G3') -- (G3);
\draw (G0) --  (G4') -- (G4);
\draw (G0) --  (G5') -- (G5);
\end{scope}

\begin{scope}[shift={(18:3)}]
\node[v, fill=orange!15] (G12) at (54:1) {\supertiny $v_3^{v_1}$};
\node[v, fill=orange!15] (G13) at (129:1) {\supertiny $v_2^{v_1}$};
\node[v, fill=orange!15] (G14) at (270:1) {\supertiny $v_5^{v_1}$};
\node[v, fill=orange!15] (G15) at (342:1) {\supertiny $v_4^{v_1}$};
\draw (G1) -- (G12);
\draw (G1) -- (G13);
\draw (G1) -- (G14);
\draw (G1) -- (G15);
\end{scope}

\begin{scope}[shift={(90:3)}]
\node[v, fill=orange!15] (G21) at (54:1) {\supertiny $v_5^{v_2}$};
\node[v, fill=orange!15] (G23) at (129:1) {\supertiny $v_4^{v_2}$};
\node[v, fill=orange!15] (G24) at (198:1) {\supertiny $v_3^{v_2}$};
\node[v, fill=orange!15] (G25) at (342:1) {\supertiny $v_1^{v_2}$};
\draw (G2) -- (G21);
\draw (G2) -- (G23);
\draw (G2) -- (G24);
\draw (G2) -- (G25);
\end{scope}

\begin{scope}[shift={(162:3)}]
\node[v, fill=orange!15] (G31) at (54:1) {\supertiny $v_2^{v_3}$};
\node[v, fill=orange!15] (G32) at (129:1) {\supertiny $v_1^{v_3}$};
\node[v, fill=orange!15] (G34) at (198:1) {\supertiny $v_5^{v_3}$};
\node[v, fill=orange!15] (G35) at (270:1) {\supertiny $v_4^{v_3}$};
\draw (G3) -- (G31);
\draw (G3) -- (G32);
\draw (G3) -- (G34);
\draw (G3) -- (G35);
\end{scope}

\begin{scope}[shift={(234:3)}]
\node[v, fill=orange!15] (G41) at (129:1) {\supertiny $v_3^{v_4}$};
\node[v, fill=orange!15] (G42) at (198:1) {\supertiny $v_2^{v_4}$};
\node[v, fill=orange!15] (G43) at (270:1) {\supertiny $v_1^{v_4}$};
\node[v, fill=orange!15] (G45) at (342:1) {\supertiny $v_5^{v_4}$};
\draw (G4) -- (G41);
\draw (G4) -- (G42);
\draw (G4) -- (G43);
\draw (G4) -- (G45);
\end{scope}

\begin{scope}[shift={(306:3)}]
\node[v, fill=orange!15] (G51) at (54:1) {\supertiny $v_1^{v_5}$};
\node[v, fill=orange!15] (G52) at (198:1) {\supertiny $v_4^{v_5}$};
\node[v, fill=orange!15] (G53) at (270:1) {\supertiny $v_3^{v_5}$};
\node[v, fill=orange!15] (G54) at (342:1) {\supertiny $v_2^{v_5}$};
\draw (G5) -- (G51);
\draw (G5) -- (G52);
\draw (G5) -- (G53);
\draw (G5) -- (G54);
\end{scope}
\end{tikzpicture}
\caption{The degree $m+1$ cover of the graph of groups. Each gray vertex group is trivial, and each orange vertex group is a copy of $\mathbb Z/2$ generated by the listed element. We write $v_i^{v_j}$ for $v_jv_iv_j$.}
\label{fig:graph of groups}
\end{figure}
    Consider a graph of groups structure of $H$ with the underlying graph an $m$-star (i.e.\ the graph with one vertex of degree $m$, and $m$ vertices of degree $1$) where each vertex group is a copy of $\mathbb Z/2$ for a vertex which is a leaf, and the non-leaf vertex group is trivial. See the left part of Figure~\ref{fig:graph of groups}. Then $K$ is the graph of groups cover represented in right part of Figure~\ref{fig:graph of groups}. Indeed, each leaf of this graph has a vertex group $\langle v_i^{v_j}\rangle_H \simeq \mathbb Z/2$. The underlying graph for $K$ is built out of $m+1$ copies of the underlying $m$-star graph of $H$. Thus $[H:K]=m+1$. 
\end{proof}

\begin{proposition}\label{prop: graph operation bipartite}
Suppose that $\Gamma$ is a bipartite graph with a vertex set partition $V(\Gamma) = L\sqcup R$.  Let $\Gamma'$ be obtained from $\Gamma$ by one of the three graph operations: vertex link double, simultaneous vertex link multiple, or edge quadruple. Then there is a natural map $\pi: \Gamma'\to\Gamma$ of graphs.
In particular, $\Gamma'$ is bipartite with a vertex set partition $V(\Gamma') = \pi^{-1}(L) \sqcup \pi^{-1}(R)$.
\end{proposition}

\begin{proof}
The map $\pi: \Gamma'\to \Gamma$ is defined on vertices by taking $v^i\to v$ for any (possibly empty) label $i$. The map extends to edges, since every edge of $\Gamma'$ is a lift of an edge in $\Gamma$. Thus the partition $V(\Gamma') = \pi^{-1}(L)\sqcup \pi^{-1}(R)$ provides the bipartite structure of $\Gamma'$. 
\end{proof}

\begin{definition}
    Let $\Gamma, \Gamma'$ be as in Proposition~\ref{prop: graph operation bipartite}. We say $v\in V(\Gamma')$ is an \emph{$L$-vertex} (respectively an \emph{$R$-vertex}) if $\pi(v) \in L$ (respectively $\pi(v)\in R$).
\end{definition}

\section{Legal systems for $n \leq 2k-1$}\label{sec: legal systems}
\subsection{The case of $n < 2k-1$} 

\begin{figure}
\centering
\begin{tikzpicture}[
    on1/.style={draw=black, fill=green!15, rounded corners,
               minimum width=2.8cm, minimum height=0.7cm, font=\small},
    off1/.style={draw=black, fill=blue!15, rounded corners,
               minimum width=2.8cm, minimum height=0.7cm, font=\small},
    on2/.style={draw=black, fill=red!15, rounded corners,
                minimum width=2.8cm, minimum height=0.7cm, font=\small},
    off2/.style={draw=black, fill=orange!15, rounded corners,
                minimum width=2.8cm, minimum height=0.7cm, font=\small},
    col/.style={font=\small\bfseries, align=center},
    >=Stealth, thick
]
 
\node[col] at (-3.5, 5.2) {Numbers};
\node[on1]  (N1) at (-3.5, 4.2) {$1, 3, 5, \ldots$};
\node[off1] (N2) at (-3.5, 2.8) {$2, 4, 6, \ldots$};
 
\node[col] at (3.5, 5.2) {$k$-subsets};
\node[on2]  (C1) at (3.5, 4.5) {$k$-subsets with $1$ except $\tau$};
\node[off2] (C2) at (3.5, 3.5) {$\tau = \{1,n,n{-}1,\ldots,n{-}k{+}2\}$};
\node[off2] (C3) at (3.5, 2.5) {$k$-subsets without $1$};
 
\draw (N1.east) -- (C1.west)
    node[midway, above, font=\tiny] {$1 \in \sigma$, mixed parity};
\draw (N1.east) -- (C2.west);
 
\draw (N2.east) -- (C1.west);
\draw (N2.east) -- (C2.west);
\draw (N2.east) -- (C3.west);
 
\node[on1, minimum width=0.8cm, minimum height=0.45cm] at (-2.25, 1.2) {};
\node[right, font=\small] at (-2.55, 1.2) {$L_0$};
\node[off1, minimum width=0.8cm, minimum height=0.45cm] at (-0.75, 1.2) {};
\node[right, font=\small] at (-1.05, 1.2) {$L_1$};
\node[on2, minimum width=0.8cm, minimum height=0.45cm] at (0.75, 1.2) {};
\node[right, font=\small] at (0.45, 1.2) {$R_0$};
\node[off2, minimum width=0.8cm, minimum height=0.45cm] at (2.25, 1.2) {};
\node[right, font=\small] at (1.95, 1.2) {$R_1$};

\end{tikzpicture}
\caption{One of the legal states in $\Lambda(n,k)$ with $k < n < 2k-1$.
}
\label{fig:small-n}
\end{figure}
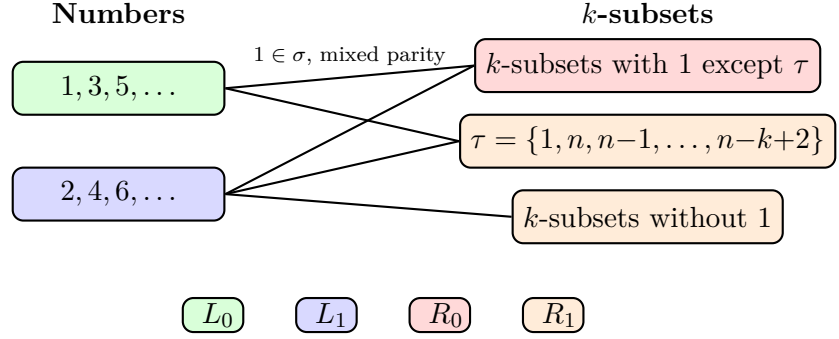

The following was proven in \cite[Example 6.2]{JankiewiczNorinWise16}, we include the proof for completeness.
 
\begin{theorem}\label{thm:small-n}
For $3\leq k < n < 2k-1$, the graph $\Lambda(n,k)$ admits a legal system.
\end{theorem}
 
\begin{proof} Fix $k\geq 3$ and $k<n<2k-1$. We use the bipartite system of moves.
Let
\[
L_0 = \{\text{odd numbers}\}, \qquad L_1 = \{\text{even numbers}\},
\]
\[
R_0 = \{\text{$k$-subsets containing } 1\}-\{\tau\}, \qquad R_1 = \{\tau\}\cup\{\text{$k$-subsets not containing }1\},
\]
where $\tau = \{1,n,n-1,\ldots,n-k+2\}$, so that $L_0 \sqcup L_1 =L$ is the set of all number vertices, $R_0 \sqcup R_1 = R$
is the set of all $k$-subset vertices.  See Figure~\ref{fig:small-n}. By Lemma~\ref{lem:bipartite moves states} it suffices to check that each $S_{ij} = L_i\cup R_j$ for $i,j\in \{0,1\}$ spans a connected subgraph of $\Gamma$.

\noindent\emph{Key observation.}
Since $n < 2k-1$, we have $k > \frac{n+1}{2}$, which means $k$ exceeds both
the number of odd elements and the number of even elements in $\{1,\ldots,n\}$.
Hence every $k$-subset of $\{1,\ldots,n\}$ must contain at least one odd and
at least one even number. In particular, \emph{no vertex in $R$ is adjacent only to vertices in $L_0$ or only to vertices in $L_1$}.
 
We first verify that in every state $S_{ij}$ there is a path between any two vertices of $L_i$. 

\noindent\emph{Case 1: $S_{00}$.} Let $x,y\in L_0$, i.e.\ $x,y$ are both odd numbers. Since $k\geq 3$, there exist $k$-subsets $\{1,x,y,...\} \neq \tau$ (when $k=3$, $n=4$, and there is unique such $k$-subset $\{1,2,3\}\neq \tau = \{1,3,4\}$). Thus $x,y$ are both adjacent to that $k$-subset, and therefore they are connected in $S_{00}$.

\noindent\emph{Case 2: $S_{01}$.} Let $x,y\in L_0$, i.e.\ $x,y$ are both odd number other than $1$. There exists a $k$-subset $\{x,y,\ldots\}$ without $1$, so any two odd numbers other than $1$ are connected in $S_{01}$. Note that the vertex $1$ is adjacent to $\tau$, meaning that $1$ is connected in $S_{01}$ to every odd number in ${n,n-1,..,n-k+2}$ (there is at least one since $n>k\geq 3)$. Thus all vertices of $L_0$ are connected in $S_{01}$.

\noindent\emph{Case 3: $S_{1j}$ for $j=0,1$} 
Let $x,y\in L_1$, i.e.\ $x,y$ are both even numbers. Then there exists a $k$-subset containing both of them, with and without $1$. Hence all even numbers are connected in $S_{1j}$.
 
We now show that $S_{ij}$ is connected for $i,j \in \{0,1\}$.
Let $\sigma, \sigma'\in R_j$. By the key observation, each contains at least one odd and one even number. Hence each has at least one neighbor in $L_0$ and at least one neighbor in $L_1$. In particular, $\sigma$ has a neighbor $x\in L_i$ and $\sigma'$ has a neighbor $x'\in L_i$. By above, $x,x'$ are connected in $S_{ij}$, so $\sigma, \sigma'$ are also connected in $S_{ij}$.
\end{proof}
 
\begin{corollary}
For $3\leq k < n < 2k-1$, the right-angled Coxeter group $W_{\Lambda(n,k)}$ virtually
algebraically fibers.
\end{corollary}

\subsection{The case of $n=2k-1$} 
Fix $k \geq 3$, and let $n = 2k-1$. By Corollary~\ref{cor: problematic vertex in Lambda} $\Lambda(2k-1, k)$ does not admit a legal system. We now show that a vertex link double of $\Lambda(2k-1, k)$  admits a legal system.

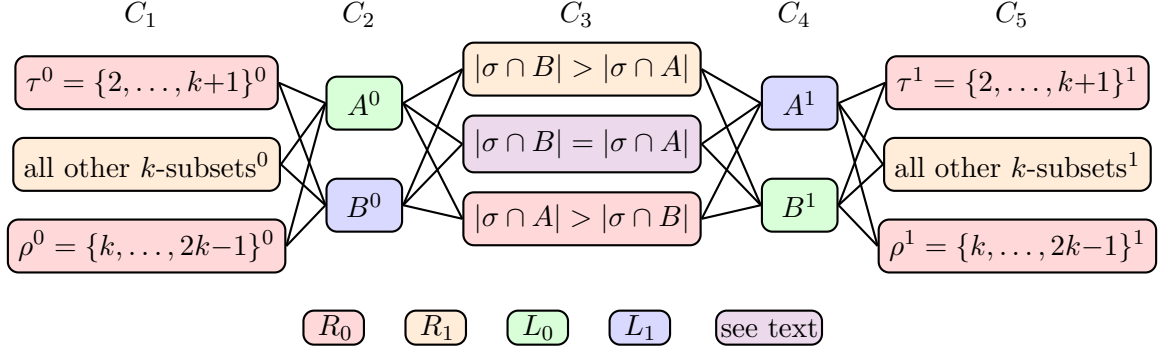
\begin{figure}
\centering
\begin{tikzpicture}[scale =0.9,
    on2/.style={draw=black, fill=green!15, rounded corners,
               minimum width=1cm, minimum height=0.7cm, font=\small},
    off2/.style={draw=black, fill=blue!15, rounded corners,
               minimum width=1cm, minimum height=0.7cm, font=\small},
    on1/.style={draw=black, fill=red!15, rounded corners,
                minimum width=1cm, minimum height=0.7cm, font=\small},
    off1/.style={draw=black, fill=orange!15, rounded corners,
                minimum width=1cm, minimum height=0.7cm, font=\small},
    arb/.style={draw=black, fill=violet!15, rounded corners,
                minimum width=3cm, minimum height=0.75cm, font=\small,
                align=center},
    col/.style={font=\small\bfseries, align=center},
    >=Stealth, thick
]
 
\def\ca{0}
\def\cb{3.2}
\def\cc{6.4}
\def\cd{9.6}
\def\ce{12.8}
 
\node[col] at (\ca, 4.6) {$C_1$
};
\node[col] at (\cb, 4.6) {$C_2$
};
\node[col] at (\cc, 4.6) {$C_3$
};
\node[col] at (\cd, 4.6) {$C_4$
};
\node[col] at (\ce, 4.6) {$C_5$
};
 
\node[on1]  (A1) at (\ca, 3.6) {$\tau^0 = \{2,\ldots,k{+}1\}^0$};
\node[off1] (A2) at (\ca, 2.4) {all other $k$-subsets$^0$};
\node[on1]  (A3) at (\ca, 1.2) {$\rho^0 = \{k,\ldots,2k{-}1\}^0$};
 
\node[on2] (B1) at (\cb, 3.3) {$A^0$};
\node[off2]  (B2) at (\cb, 1.8) {$B^0$};
 
\node[off1]  (C1) at (\cc, 3.8) {$|\sigma\cap B|>|\sigma\cap A|$};
\node[arb] (C2) at (\cc, 2.7) {$|\sigma\cap B|=|\sigma\cap A|$};
\node[on1] (C3) at (\cc, 1.6) {$|\sigma\cap A|>|\sigma\cap B|$};
 
\node[off2]  (D1) at (\cd, 3.3) {$A^1$};
\node[on2] (D2) at (\cd, 1.8) {$B^1$};
 
\node[on1]  (E1) at (\ce, 3.6) {$\tau^1 = \{2,\ldots,k{+}1\}^1$};
\node[off1] (E2) at (\ce, 2.4) {all other $k$-subsets$^1$};
\node[on1]  (E3) at (\ce, 1.2) {$\rho^1 = \{k,\ldots,2k{-}1\}^1$};
 
\draw (A1.east) -- (B1.west);
\draw (A1.east) to (B2.west);
\draw (A2.east) -- (B1.west);
\draw (A2.east) to (B2.west);
\draw (A3.east) -- (B2.west);
\draw (A3.east) to (B1.west);
 
\draw (B1.east) -- (C3.west);
\draw (B2.east) -- (C1.west);
\draw (B1.east) --(C1.west);
\draw (B1.east) --(C2.west);
\draw (B2.east) --(C3.west);
\draw (B2.east) --(C2.west);
 
\draw (C1.east) -- (D1.west);
\draw (C3.east) -- (D2.west);
\draw (C1.east) --(D2.west);
\draw (C2.east) --(D1.west);
\draw (C2.east) --(D2.west);
\draw (C3.east) --(D1.west);
 
\draw (D1.east) -- (E1.west);
\draw (D1.east) to (E3.west);
\draw (D2.east) -- (E3.west);
\draw (D2.east) to (E1.west);
\draw (D1.east) to (E2.west);
\draw (D2.east) to (E2.west);
 
\begin{scope}[shift={(5,0)}]
\node[on1, minimum width=0.8cm, minimum height=0.45cm] at (-2.25, 0) {};
\node[right, font=\small] at (-2.65, 0) {$R_0$};
\node[off1, minimum width=0.8cm, minimum height=0.45cm] at (-0.75, 0) {};
\node[right, font=\small] at (-1.15, 0) {$R_1$};
\node[on2, minimum width=0.8cm, minimum height=0.45cm] at (0.75, 0) {};
\node[right, font=\small] at (0.35, 0) {$L_0$};
\node[off2, minimum width=0.8cm, minimum height=0.45cm] at (2.25, 0) {};
\node[right, font=\small] at (1.85, 0) {$L_1$};
\node[arb, minimum width=1.4cm, minimum height=0.45cm]  at (4.15, 0) {};
\node[right, font=\small] at (3.27, 0) {see text};
 \end{scope}

\end{tikzpicture}
\caption{The structure of the vertex link double of $\Lambda(2k-1, k)$. 
}
\label{fig:columns}
\end{figure}
 
\begin{theorem}\label{thm:main}
For $k \geq 3$, the graph $\Gamma$ obtained by vertex link doubling $\Lambda(2k-1,k)$
along any number vertex admits a legal system.
\end{theorem}
 
\begin{proof}
Fix $n = 2k-1$, $k \geq 3$. Without loss of generality we  take a vertex link double along the vertex $1$. We split the vertices of the resulting graph $\Gamma$ in five sets  $C_1, C_2, C_3, C_4, C_5$ as in Figure~\ref{fig:columns}, with edges only between vertices in adjacent sets.
\begin{itemize}
    \item $C_1 = \{\sigma^0 \mid \sigma\subseteq \{2, \dots, 2k-1\}, |\sigma| =k\}$, i.e.\ one copy of the $k$-subsets of $\{2, \dots, 2k-1\}$;
    \item $C_2 = \{2^0,\ldots,2k-1^0\}$;
    \item  $C_3 = \{\sigma \mid 1\in \sigma\subseteq \{1, \dots, 2k-1\}, |\sigma| =k\}$, i.e. the $k$-subsets of $\{1,2,...,2k-1\}$ containing $1$;
    \item $C_4 = \{2^1,\ldots,2k-1^1\}$;
    \item $C_5 = \{\sigma^1 \mid \sigma\subseteq \{2, \dots, 2k-1\}, |\sigma| =k\}$, i.e.\ the other copy of the $k$-subsets of $\{2, \dots, 2k-1\}$.
\end{itemize}
By Proposition~\ref{prop: graph operation bipartite} $\Gamma$ is bipartite, with the bipartite structure $L\sqcup R$ where $L = C_2\cup C_4$ and $R = C_1\cup C_3 \cup C_5$. We use the bipartite system of moves. We now construct an explicit orbit of states $L_i\cup R_j$ (see Figure~\ref{fig:columns} for one of them) and verify that each state is legal.
 
Write $A = \{2, \ldots, k\}$ and
$B = \{k+1, \ldots, 2k-1\}$, so $\{2,\ldots,2k-1\} = A \sqcup B$ with
$|A| = |B| = k-1$. We set $L_0 = A^0\cup B^1$ and $L_1 = A^1\cup B^0$. Clearly $L_0 \sqcup L_1 = L$.

Let $\tau = \{2,3,\ldots,k+1\}$ and $\rho = \{k,k+1,\ldots,2k-1\}$. 
We set
$$R_0 = \{\tau^0, \rho^0, \tau^1, \rho^1\} \cup \{\sigma \in C_3 : |\sigma\cap A|>|\sigma\cap B|\} \cup T_0$$ 
and 
$$R_1 = (C_{1}-\{\tau^0, \rho^0\}) \cup (C_{5}-\{\tau^1, \rho^1\}) \cup \{\sigma \in C_3 : |\sigma\cap A|<|\sigma\cap B|\} \cup T_1,$$
where $T_0 = T = \{\sigma \in C_3: |\sigma\cap A| = |\sigma \cap B|\}$ and $T_1 =\emptyset$ when $k>3$ (note that $T=\emptyset$ for $k$ even). When $k=3$, we set $T_0 = \{\{1,2,4\},\{1,3,5\}\}$ and $T_1 = \{\{1,2,5\},\{1,3,4\}\}$.
It is clear that $R_0\sqcup R_1 = R$.

By Lemma~\ref{lem:bipartite moves states} it suffices to show that each $S_{ij} = L_i\cup R_j$ for $i,j\in\{0,1\}$ spans a connected subgraph of $\Gamma$.

Since only vertices in  adjacent columns are joined by edges, we establish connectivity
within each column via paths through neighboring columns.

\medskip
\noindent\emph{Step 1: Connecting vertices from the same column $C_2$ or $C_4$}
We first check that in each state $S_{ij}$, any two vertices either from column $C_2$ or from column $C_4$ are connected to one another through paths in $S_{ij}$.
Since $A=\{2,\dots, k\}\subseteq \tau$, the vertices of $A^i$ are all adjacent to $\tau^i$. In particular, any two vertices of $A^i$ are connected in $S_{i0}$.

Any two vertices of $A_i$ have a common neighbor $\sigma \in R_1$ where $1\notin \sigma$ (so $\sigma\in C_1$ or $C_5$). This shows that any two vertices of $A^i$ are connected in $S_{i1}$.

Similarly, the vertices in $B^j$ are all adjacent to $\rho^j$, and therefore connected in $S_{(1-j)0}$.
Also similarly, any two vertices of $B_j$ have a common  neighbor $\sigma \in R_1$ where $1\notin \sigma$, and so are connected in $S_{(1-j)1}$.

\medskip
\noindent\emph{Step 2: Connecting vertices from the same column $C_1$ or $C_5$}
Since $\tau\cap\rho = \{k,k+1\}$, for $i=0,1$, the vertices $\tau^i$ and $\rho^i$ are both adjacent to the vertex $k^i\in A^i$, and they are both adjacent to the vertex $(k+1)^i\in B^i$. Thus $\tau^i, \rho^i\in R_0$ are connected in both $S_{00}$ and $S_{10}$.

Every $k$-subset $\sigma$ of the $2k-2$ element set $\{2,\ldots,2k-1\}$ must contain at least one element of $A$ and at least one element of $B$. Hence every vertex $\sigma^0 \in R_1\cap C_1$ is adjacent
to some element of $A^0$, and is also adjacent to some element of $B^0$. 
Since by previous step any two vertices of $A^0$ are connected in $S_{01}$, we deduce that any two vertices of $R_1\cap C_1$ are connected in $S_{01}$. Similarly we argue that any two vertices of $R_1\cap C_1$ are connected in $S_{11}$. 
The same argument applies to vertices of $R_1\cap C_5$.
\medskip

Thus we have shown that the restriction of $S_{ij}$ to $C_1\cup C_2$ or to $C_4\cup C_5$ spans a connected subgraph. We now consider the middle column $C_3$ to show that the entire $S_{ij}$ spans a connected subgraph.

\medskip
\noindent\emph{Step 3: Connecting vertices from the column $C_3$.}
Every $k$-subset $\sigma$ containing $1$ contains $k-1\geq 2$ more elements, so since $L_0 = A_0 \cup B_1$ and $L_1 = A_1 \cup B_0$, each $\sigma\in C_3$ is adjacent to some vertex of $L_0$ and to some vertex of $L_1$. 
Since $|A| = |B| = k-1$, $\sigma_A = A\cup \{1\}$ and $\sigma_{B} = B\cup\{1\}$ are vertices in $C_3$. Note that these are the only vertices of $C_3$ which are not adjacent to at least one vertex in each of $A^0, B^0, A^1, B^1$.

When $k>3$, then $C_3\cap R_j$ contains at least one vertex other than $\sigma_A$, $\sigma_B$. In particular, this vertex connects previously considered connected subgraphs of $S_{ij}$. Moreover, $\sigma_A$ shares a common neighbor in each $A^i$ with some other vertex in $C_3\cap R_0$, and similarly $\sigma_B$ shares a common neighbor in each $B^i$ with some other vertex in $C_3\cap R_1$.
This proves that each $S_{ij}$ spans a connected subgraph of $\Gamma$, when $k>3$.

Finally, we consider the case of $k=3$, which is illustrated in Figure~\ref{fig:k=3,n=5}. We directly check that each $S_{ij}$ restricted to $C_{2}\cup C_3 \cup C_4$ spans a connected subgraph.
\begin{figure}
\centering
\begin{tikzpicture}[
 superthick/.style={line width=2.5pt},
    v/.style={draw=black, fill=blue!15, circle, minimum size=0.1cm, font=\small},
    >=Stealth, thick
    on2/.style={draw=black, fill=green!15, rounded corners,
               minimum width=1cm, minimum height=0.7cm, font=\small},
    off2/.style={draw=black, fill=blue!15, rounded corners,
               minimum width=1cm, minimum height=0.7cm, font=\small},
    on1/.style={draw=black, fill=red!15, rounded corners,
                minimum width=1cm, minimum height=0.7cm, font=\small},
    off1/.style={draw=black, fill=orange!15, rounded corners,
                minimum width=1cm, minimum height=0.7cm, font=\small},
    arb/.style={draw=black, fill=violet!15, rounded corners,
                minimum width=2.5cm, minimum height=0.75cm, font=\small,
                align=center},
    col/.style={font=\small\bfseries, align=center},
    >=Stealth, thick
]
 
\def\ca{0}
\def\cb{3.2}
\def\cc{6.4}
\def\cd{9.6}
\def\ce{12.8}

\node[v, fill=blue!15] (50) at (\cb, 0) {$5^0$};
\node[v, fill=blue!15] (40) at (\cb, 1) {$4^0$};
\node[v, fill=green!15] (30) at (\cb, 2) {$3^0$};
\node[v, fill=green!15] (20) at (\cb, 3) {$2^0$};

\node[on1]  (123) at (\cc, 4) {$\{1,2,3\}$};
\node[on1]  (124) at (\cc, 3) {$\{1,2,4\}$};
\node[off1]  (125) at (\cc, 2) {$\{1,2,5\}$};
\node[off1]  (134) at (\cc, 1) {$\{1,3,4\}$};
\node[on1]  (135) at (\cc, 0) {$\{1,3,5\}$};
\node[off1]  (145) at (\cc, -1) {$\{1,4,5\}$};

\node[v, fill=green!15] (51) at (\cd, 0) {$5^1$};
\node[v, fill=green!15] (41) at (\cd, 1) {$4^1$};
\node[v, fill=blue!15] (31) at (\cd, 2) {$3^1$};
\node[v, fill=blue!15] (21) at (\cd, 3) {$2^1$};

\draw (20.east) -- (123.west);
\draw (20.east) -- (124.west);
\draw (20.east) --(125.west);
\draw (30.east) -- (123.west);
\draw (30.east) -- (134.west);
\draw (30.east) --(135.west);
\draw[superthick] (40.east) -- (124.west);
\draw (40.east) -- (134.west);
\draw (40.east) --(145.west);
\draw (50.east) -- (125.west);
\draw[superthick] (50.east) -- (135.west);
\draw (50.east) --(145.west); 
\draw[superthick] (21.west) -- (123.east);
\draw[superthick] (21.west) -- (124.east);
\draw (21.west) --(125.east);
\draw[superthick] (31.west) -- (123.east);
\draw (31.west) -- (134.east);
\draw[superthick] (31.west) --(135.east);
\draw (41.west) -- (124.east);
\draw (41.west) -- (134.east);
\draw (41.west) --(145.east);
\draw (51.west) -- (125.east);
\draw (51.west) -- (135.east);
\draw (51.west) --(145.east); 
 
\end{tikzpicture}
\caption{The thickened edges belong to the connected subgraph spanned by $S_{10}$. One can directly check that $S_{00}, S_{01}, S_{11}$ also span connected subgraphs.}
\label{fig:k=3,n=5}
\end{figure}
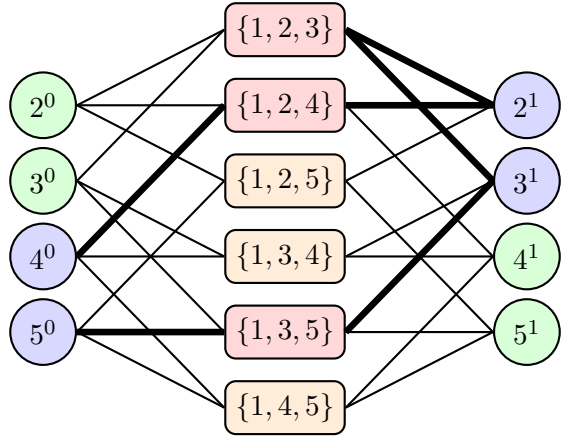
\end{proof}

\begin{corollary}
For $k\geq 3$, the right-angled Coxeter group $W_{\Lambda(2k-1,k)}$ virtually
algebraically fibers.
\end{corollary}

\section{Obstruction for $n \geq 2k$}\label{sec: obstructions}
\subsection{Problematic vertices}
\begin{definition}
Let $\Gamma = (V,E)$ be a simplicial graph, and let $S\subseteq V$ be a state. A vertex $u \in V(\Gamma)$ is \emph{problematic with respect to $S$} if all neighbors of $u$ lie in $S$, or all lie in $V-S$. 
\end{definition}

\begin{lemma}\label{lem: problematic vertex obstruction}
 Let $\Gamma$ be a graph where for every $v\in V(\Gamma)$ there exist at least two vertices distinct from $v$ which are not contained in $N(v)$.
If for every state $S$ of $\Gamma$, there exists a problematic vertex with respect to $S$, then $\Gamma$ does not admit a legal system.
\end{lemma}
\begin{proof}
    Suppose $\{M_v:v\in V(\Gamma)\}$ is a system of moves, and $\mathcal S$ is an orbit of legal states. Take a state $S\in \mathcal S$, and let $v$ be a problematic vertex with respect to $S$. 
    This means that all neighbors of $v$ are either in $S$ or $V(\Gamma)-S$. Without loss of generality, assume $N(v)\subseteq S$.
    Since, $N(v)\cap M_v=\emptyset$, we also have $N(v)\subseteq M_v\cdot S$. But $v$ belongs to only one of $S, M_v\cdot S$, without loss of generality assume $v\notin S$. 
    Now note that if there exists another vertex $v'\notin S$, then there is no path between $v$ and $v'$ which only passes through vertices of $V-S$. Thus all other vertices of $V(\Gamma)$ are in $S$. 
    By assumption there exist $u_1, u_2\notin \{v\}\cup N(v)$. For each $i=1,2$ we claim that $M_{u_i} = \{v,u_i\}$. Indeed, since $u_i\notin N(v)$, all the neighbors of $u_i$ are in $S$. Thus $u_i\notin M_{u_i}\cdot S$ but $N(u_i) \subseteq S$, which again implies that $V-M_{u_i}\cdot S = \{u_i\}$, and so $M_{u_i} = \{v,u_i\}$. Thus $M_{u_1}M_{u_2}\cdot S = V-\{v, u_1, u_2\}$ which is not legal, as the graph on $\{v, u_1, u_2\}$ is not connected.
\end{proof}

\begin{remark}
Lemma~\ref{lem: problematic vertex obstruction} does not assume that the system of moves is bipartite. We note however that for a bipartite system of moves and a state $S$, a vertex $u$ is problematic with respect to $S$ if and only if $u$ is problematic with respect to any state in the orbit of $S$.
\end{remark}

The assumptions of Lemma~\ref{lem: problematic vertex obstruction} are satisfied in every bipartite graph where each side of the partition has at least three vertices. In particular, we can apply it to $\Lambda(n,k)$ with $k\geq 3$ and $n\geq 2k-1$.

\begin{corollary}[{See also \cite[Prop 6.1]{JankiewiczNorinWise16}}]\label{cor: problematic vertex in Lambda} Let $k\geq 3$ and $n\geq 2k-1$. For every state $S$ of $\Lambda(n,k)$ there exists a problematic vertex $r_0\in R$. 
In particular, $\Lambda(n,k)$  does not admit a legal system.
\end{corollary}
\begin{proof}
Since $n\geq 2k-1$, for every state $S$ at least $k$ vertices from $L$ are either in $S$ or in $V(\Gamma)-S$. Thus there exists a problematic $k$-subset vertex $r_0\in R$ with respect to $S$. By Lemma~\ref{lem: problematic vertex obstruction} $\Lambda(n,k)$ does not admit a legal system.
\end{proof}

\begin{definition}\label{def:H}
Let $\Gamma = (V,E)$ be a bipartite graph with the partition $V(\Gamma) = L\sqcup R$. For a state $S\subseteq V$, write $P_R(S) \subseteq
V(\Gamma)$ for the set of \emph{problematic $R$-vertices} of $S$, i.e.\ the vertices in $R$ which are 
problematic with respect to $S$. 

We say $\Gamma$ has
\emph{dispersed problematic $R$-vertices}
if for every state $S$ of $\Gamma$, no vertex
$u \in V(\Gamma)$ is adjacent to every element of $P_R(S)$, i.e.
\[
P_R(S) \not\subseteq N(u) \quad \text{for every } u \in V(\Gamma).
\]
\end{definition}
In particular, $P_R(S)\neq\emptyset$.

\begin{proposition}\label{prop: dispersed problematic in Lambda}
    Let $n\geq 2k$ and $k\geq 3$. Then $\Lambda(n,k)$ has dispersed problematic $R$-vertices.
\end{proposition}
\begin{proof} Let $S$ be a state of $\Lambda(n,k)$, and let $L = \{1, \dots, n\} = L_0\sqcup L_1$ be the induced partition of $L$, i.e.\ $L_0 = L\cap S$ and $L_1= L\cap (V-S)$. Since $n\geq 2k$, there exist at least two $k$-subsets of $L$ entirely contained in $L_0$ or in $L_1$. If $|L_0| = |L_1| = k$, there are exactly two such $k$-subsets, namely $L_0$ and $L_1$, and since $L_0\cap L_1 = \emptyset$, the vertices $L_0,L_1\in R$ have no common neighbors. Otherwise at least one of $L_0, L_1$, say $L_0$, has at least $k+1$ elements. Then any of the $k$-subsets $\sigma$ of $L_0$ belongs to $P_R(S)$. Since $$\bigcap_{\substack{\sigma\subseteq L_0\\ |\sigma|=k}} \sigma = \emptyset$$
there is no $u\in L$ such that $P_{R}(S)\subseteq N(u)$ (and of course no $u\in R$ as $\Gamma$ is bipartite). Thus $\Lambda(n,k)$ has dispersed problematic $R$-vertices.
\end{proof}

\subsection{Persistence of the dispersed problematic vertices under graph operations}
We check that each of the graph operations defined in Section~\ref{sec: graph operations} preserves the property of having dispersed problematic $R$-vertices. We start with the vertex link double.
\begin{proposition}
\label{prop: problematic under link double}
Let $\Gamma$ be a bipartite graph which has dispersed problematic $R$-vertices, 
and let $v \in V(\Gamma)$. 
The vertex link double $\Gamma'$ of $\Gamma$ along $v$ also
has dispersed problematic $R$-vertices.
\end{proposition}

\begin{proof}
Let $V(\Gamma) = L\sqcup R$ be the partition providing the bipartite structure. We consider two cases: $v$ has type $L$, and $v$ has type $R$.
 
\noindent\emph{Case: $v\in L$.} 
As in Definition~\ref{defn: link double} the vertex set of $\Gamma'$ decomposes into 
$$ V(\Gamma') =  (R-N(v))^0\sqcup (L-\{v\})^0\sqcup N(v) \sqcup (L-\{v\})^1\sqcup (R-N(v))^1.$$
See Figure~\ref{fig:five-column-generic}. 
\begin{figure}
\centering
\begin{tikzpicture}[
    outer/.style={draw=black, fill=orange!15, rounded corners,
                  minimum width=2cm, minimum height=0.5cm, font=\small,
                  align=center},
    inner/.style={draw=black, fill=gray!15, rounded corners,
                  minimum width=2cm, minimum height=0.5cm, font=\small,
                  align=center},
    mid/.style={draw=black, fill=orange!15, rounded corners,
                minimum width=2cm, minimum height=0.5cm, font=\small},
    col/.style={font=\small\bfseries, align=center},
    >=Stealth, thick
]
\def\ca{0}
\def\cb{2.9}
\def\cc{5.6}
\def\cd{8.3}
\def\ce{11.2}
\node[outer] (A1) at (\ca, 0) {$(R - N(v))^0$};
\node[inner] (B1) at (\cb, 0) {$(L - \{v\})^0$};
\node[mid]   (C1) at (\cc, 0) {$N(v)$};
\node[inner] (B2) at (\cd, 0) {$(L - \{v\})^1$};
\node[outer] (A2) at (\ce, 0) {$(R - N(v))^1$};
\draw (A1.east) -- (B1.west);
\draw (B1.east) -- (C1.west);
\draw (C1.east) -- (B2.west);
\draw (B2.east) -- (A2.west);
\end{tikzpicture}
\caption{The structure of the vertex link double of a
bipartite graph $\Gamma$ along $v \in L$. Vertex link doubling along $v\in R$ yields a similar graph, with $L$ and $R$ swapped.}
\label{fig:five-column-generic}
\end{figure}
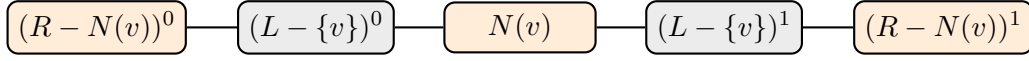
Let $S'$ be an arbitrary state of $\Gamma'$. 
Let $\Gamma^i$ be the subgraph of $\Gamma'$ spanned by $(R-N(v))^i\sqcup (L-\{v\})^i\sqcup N(v)$, which is naturally isomorphic to the subgraph of $\Gamma$ spanned by $V(\Gamma)-\{v\}$. Note that $S'$ restricts to a state $S^i$ of $\Gamma^i$, which can be extended to a state $S$ of $\Gamma$ by choosing $v\in S$ or $v\notin S$.

Since $\Gamma$ has dispersed problematic $R$-vertices, we know that for any state $S$ of $\Gamma$, there exists $w\in P_R(S)$ such that $w\notin N(v)$. 
Independent of how the state $S^i$ is extended to $S$, there exists $w^i\in P_R(S')$ such that $w\notin N(v)$, i.e.\ $w^i\in (R-N(v))^i$. Since the distance between $w^0$ and $w^1$ in $\Gamma'$ is at least $4$, we conclude that $\Gamma'$ has dispersed problematic $R$-vertices.

\noindent\emph{Case: $v\in R$.} Now the vertex set of $\Gamma'$ decomposes into 
$$ V(\Gamma') =  (L-N(v))^0\sqcup (R-\{v\})^0\sqcup N(v) \sqcup (R-\{v\})^1\sqcup (L-N(v))^1.$$ We will show that for every $u \in V(\Gamma')$, there exists $w \in P_R(S')$ with $(u,w)\notin E(\Gamma')$. Clearly, it suffices to consider only $L$-vertices $u$.
Suppose $u\in (L-N(v))^i$. Then the state $S^{(1-i)}$ of $\Gamma^{(1-i)}$ extends to a state of $\Gamma$ for which there must exist problematic $R$-vertices, and at least one of them $w$ is not equal to $v$. Then $w^{(1-i)}$ is problematic with respect to $S'$, and not contained in $N(u)$. Suppose that $u\in N(v)$. 
Similarly as in the previous case, restrict the state $S'$ to a state $S^i$ of $\Gamma^i$, which can again be extended to a state of $\Gamma$ by choosing $v\in S$ or $v\notin S$. Note that $v$ might or might not be problematic with respect to $S^i$ extended to $\Gamma$. 
Since $\Gamma$ has dispersed problematic $R$-vertices, there exists $w^i$ which is problematic with respect to the extended $S^i$, and which is not adjacent to $u$.
\end{proof}

Next we consider the simultaneous vertex link multiple. 
\begin{proposition}
\label{prop: problematic under link multiple}
Let $\Gamma$ be a bipartite graph which has dispersed problematic $R$-vertices, 
and let $v_1, v_2, \dots, v_m \in V(\Gamma)$ be vertices with disjoint stars, where $m\geq 2$. 
The simultaneous vertex link multiple $\Gamma'$ of $\Gamma$ along $v_1, \ldots, v_m$ also has dispersed problematic $R$-vertices.
\end{proposition}

\begin{proof}
Let $S'$ be any state of $\Gamma'$. For each $1 \leq i \leq m$, let $\Gamma^i$ be a subgraph spanned by $B^i\sqcup N(v_i)$, using the notation of Definition~\ref{defn: link multiple}. Note that $\Gamma^i$ is isomorphic to the subgraph of $\Gamma$ spanned by $V(\Gamma)-\{v_i\}$. Consider the restriction of the state $S'$ to $\Gamma^i$. Since $\Gamma$ has dispersed problematic $R$-vertices with respect to any state, for each $i$, $B^i$ contains a problematic $R$-vertex with respect to $S'$. No two such vertices belonging to different $B^i$'s share a common neighbor, so $\Gamma'$ has dispersed problematic $R$-vertices.
\end{proof}

Finally, we consider the edge quadruple. We first analyze the structure of an edge quadruple of a bipartite graph. Let $\Gamma$ be a bipartite graph with a vertex set partition $V(\Gamma) = L\sqcup R$. Let $e=(u,v)\in E(\Gamma)$ where $u\in L$ and $v\in R$. Using the notation of Definition~\ref{defn: edge quadruple}, when $\Gamma$ is bipartite, then $N_{uv} = \emptyset$, $N_u = N(u)$, and $N_v = N(v)$. We can also further decompose $B = V(\Gamma) - (\{u,v\}\cup N(u)\cup N(v))$ into $B_L \sqcup B_R$ where $B_L = B\cap L = L - \left(\{u\}\cup N(v)\right)$ and $B_R = B\cap R = R- \left(\{v\} \cup N(u)\right)$. See Figure~\ref{fig:quadrupling}.

\begin{figure}
\centering
\begin{tikzpicture}[
    num/.style={draw=black, fill=blue!15, circle,
                minimum size=1.1cm, font=\small\bfseries},
    cfg/.style={draw=black, fill=orange!20, circle,
                minimum size=1.1cm, font=\small\bfseries},
    prob/.style={draw=red!70!black, fill=red!15, circle,
                 minimum size=1.1cm, font=\small\bfseries,
                 line width=1.2pt},
    >=Stealth, thick
]

\node[num] (X1) at ( 0,  1.5) {$N_u^0$};
\node[num] (X2) at ( 0, -1.5) {$N_u^1$};
\node[cfg] (Y1) at (-1.5,  0) {$N_v^0$};
\node[cfg] (Y2) at ( 1.5,  0) {$N_v^1$};

\draw (X1) -- (Y1);
\draw (X1) -- (Y2);
\draw (X2) -- (Y1);
\draw (X2) -- (Y2);

\node[cfg] (WTL) at (-1.5,  3) {$B_L^{00}$};
\node[cfg] (WTR) at ( 1.5,  3) {$B_L^{10}$};
\node[cfg] (WBL) at (-1.5, -3) {$B_L^{01}$};
\node[cfg] (WBR) at ( 1.5, -3) {$B_L^{11}$};

\draw (X1) -- (WTL);
\draw (X1) -- (WTR);
\draw (X2) -- (WBL);
\draw (X2) -- (WBR);

\node[num] (ZTL) at (-3,  1.5) {$B_R^{00}$};
\node[num] (ZTR) at ( 3,  1.5) {$B_R^{10}$};
\node[num] (ZBL) at (-3, -1.5) {$B_R^{01}$};
\node[num] (ZBR) at ( 3, -1.5) {$B_R^{11}$};

\draw (ZTL) -- (WTL);
\draw (ZTR) -- (WTR);
\draw (ZBL) -- (WBL);
\draw (ZBR) -- (WBR);

\draw (ZTL) -- (Y1);
\draw (ZBL) -- (Y1);
\draw (ZTR) -- (Y2);
\draw (ZBR) -- (Y2);
\end{tikzpicture}
\caption{The quadrupled graph.}
\label{fig:quadrupling}
\end{figure}
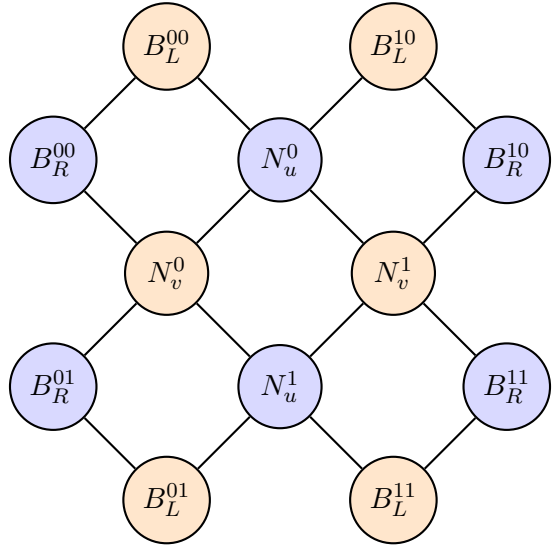

\begin{proposition}
\label{prop: problematic under quadruple}
Let $\Gamma$ be a bipartite graph which has dispersed problematic $R$-vertices, 
and let $e=(u,v) \in E(\Gamma)$ be an edge. 
The edge quadruple $\Gamma'$ of $\Gamma$ along $e$ also has dispersed problematic $R$-vertices.
\end{proposition}
\begin{proof}
Let $S'$ be any state of $\Gamma'$. For each $i,j\in \{0,1\}$, the subgraph $\Gamma^{ij}$ spanned by $B_L^{ij}\sqcup B_R^{ij} \sqcup N_v^i \sqcup N_u^j$ is isomorphic to the subgraph of $\Gamma$ spanned by $V(\Gamma)-\{u,v\}$.
Since $\Gamma$ has dispersed problematic $R$-vertices for every state, there exists at least one problematic vertex $w\not\in N(u)$.
Restricting the state $S'$ to $\Gamma^{ij}$ and viewing it as a subgraph of $\Gamma$ yields a problematic vertex in $B_R^{ij}$ with respect to $S'$. 
Since no vertex in $B_R^{01}$ shares a neighbor with a vertex from $B_{R}^{10}$, we deduce that $\Gamma'$ has dispersed problematic $R$-vertices.
\end{proof}

Finally we can prove the main result of this section.
\begin{theorem}\label{thm: no legal system operation}
For $n \geq 2k$ and $k \geq 3$, no graph obtained from $\Lambda(n,k)$ by any
sequence of vertex link doubling, simultaneous vertex link multiplying, and edge quadrupling admits a legal system.
\end{theorem}

\begin{proof}
Let $\Lambda(n,k)  =\Gamma_0, \Gamma_1, \dots,\Gamma_m = \Gamma'$ be a sequence where each $\Gamma_{i+1}$ is obtained from $\Gamma_i$ by one of the three graph operations. We show by induction that each $\Gamma_i$ is bipartite and has dispersed problematic $R$-vertices. Bipartiteness follows from Proposition~\ref{prop: graph operation bipartite}.
For the dispersed problematic $R$-vertices, the base case follows from Proposition~\ref{prop: dispersed problematic in Lambda}. The inductive step follows from Proposition~\ref{prop: problematic under link double}, Proposition~\ref{prop: problematic under link multiple}, and Proposition~\ref{prop: problematic under quadruple}. Thus $\Gamma'$ does not admit a legal system by Lemma~\ref{lem: problematic vertex obstruction}.
\end{proof}

\section{Infinite generation of the kernel of $\chi: W_{\Gamma}'\to\mathbb Z$}\label{sec: infinite generations}

\subsection{Twisted homology criterion}
Let $X$ be a connected CW-complex, $\chi\colon\pi_1X\to\mathbb Z$ a surjective homomorphism, and $\widehat X\to X$ the associated infinite cyclic cover, so that $\pi_1\widehat X = \ker\chi$. 
The complex $\widehat X$ admits a deck transformation action by $\mathbb Z$, and so the cellular chain complex $C_*(\widehat X; \mathbb Q)$ can be viewed as a complex of free $\mathbb Q[t^{\pm 1}]$-modules, freely generated by (lifts of) cells of $X$, and where the boundary maps are $\mathbb Q[t^{\pm1}]$-linear. For reference, see e.g.~\cite[VII.7]{BrownBook82and94}. Define
$$C^{\chi}_*(X;\mathbb Q(t)) := C_*(\widehat X;\mathbb Q)\otimes_{\mathbb Q[t^{\pm1}]}\mathbb Q(t),$$
with boundary maps $\partial_*\otimes1$. Thus $C^{\chi}_n(X;\mathbb Q(t))$ is a $\mathbb Q(t)$-vector space with basis the $n$-cells of $X$, and its boundary maps are twisted by powers of $t$.

Let $(C^*_{\chi}(X;\mathbb Q(t)),\delta^*)$ be the dual cochain complex, i.e.
$$C^n_{\chi}(X;\mathbb Q(t)) = \operatorname{Hom}_{\mathbb Q(t)}\big(C^{\chi}_n(X;\mathbb Q(t)),\mathbb Q(t)\big),$$
with $\delta^n$ the transpose of $\partial_{n+1}$. We write
$$H_n^{\chi}(X;\mathbb Q(t)) = H_n\big(C_*^{\chi}(X;\mathbb Q(t))\big),\qquad H^n_{\chi}(X;\mathbb Q(t)) = H^n\big(C^*_{\chi}(X;\mathbb Q(t))\big).$$
Since $\mathbb Q(t)$ is a field, by the universal coefficient theorem we have
$$H^n_{\chi}(X;\mathbb Q(t))\cong\operatorname{Hom}_{\mathbb Q(t)}\big(H^{\chi}_n(X;\mathbb Q(t)),\mathbb Q(t)\big),$$
so in particular one vanishes if and only if the other does.

\begin{proposition}\label{prop: cocycle criterion}
Let $X$ be a connected CW-complex and $\chi\colon\pi_1X\to\mathbb Z$ a surjective homomorphism. If $\ker\chi$ is  finitely generated, then $H_1^{\chi}(X;\mathbb Q(t))=0$ (equivalently $H^1_{\chi}(X;\mathbb Q(t))=0$).
\end{proposition}

\begin{proof}
Since $\mathbb Q(t)$ is a localization of $\mathbb Q[t^{\pm1}]$, it is flat as a $\mathbb Q[t^{\pm1}]$-module, so 
$$H_1^{\chi}(X;\mathbb Q(t))\cong H_1(\widehat X;\mathbb Q)\otimes_{\mathbb Q[t^{\pm1}]}\mathbb Q(t).$$
Suppose $\ker\chi$ is finitely generated. Then $H_1(\widehat X;\mathbb Q) = H_1(\ker\chi;\mathbb Q)$ is finite-dimensional over $\mathbb Q$, say of dimension $d$. For $m\in H_1(\widehat X;\mathbb Q)$, the $d+1$ elements $m,tm,\dots,t^dm$ are linearly dependent over $\mathbb Q$, so $p(t)m = 0$ for some nonzero $p\in\mathbb Q[t^{\pm1}]$. Hence for every $f\in\mathbb Q(t)$,
$$m\otimes f = m\otimes p(t)\,p(t)^{-1}f = p(t)m\otimes p(t)^{-1}f = 0,$$
so $H_1^{\chi}(X;\mathbb Q(t)) = 0$. By the universal coefficient theorem $H^1_{\chi}(X;\mathbb Q(t)) = 0$.
\end{proof}

\subsection{Legal systems versus finite generation for bipartite system of moves}
Let $\Gamma$ be a bipartite graph, with the vertex set partition $V(\Gamma) = L\sqcup R$.
    Let $L= L_0\sqcup L_1, R = R_0\sqcup R_1$ yield an orbit of states $\mathcal S= \{S_{ij}: i,j\in \{0,1\}\}$ under the bipartite moves. Let $\Gamma_{ij}$ be the subgraph of $\Gamma$ spanned by $S_{ij} = L_i\cup R_j$. Let $\chi:W_{\Gamma}'\to \mathbb Z$ be the associated homomorphism, which is non-trivial when each $S\in \mathcal S$ is non-empty.

In the next lemma we consider the complex $\overline X_{\Gamma}$ defined in Remark~\ref{rem: 4vertex complex}, and the associated homomorphism $\overline \chi: \pi_1 \overline X_{\Gamma}\to \mathbb Z$ extending $\chi$. Recall that $\overline X_{\Gamma}$ has:
\begin{itemize}
    \item $4$ vertices: $x_{00}, x_{01}, x_{10}, x_{11}$;
    \item $2|V(\Gamma)|$ edges: for each $\ell\in L$ and $j=0,1$ an oriented edge $e_{\ell, j}$ from $x_{0j}$ to $x_{1j}$, and for each $r\in R$ and $i=0,1$ an oriented edge $e_{i, r}$ from $x_{i0}$ to $x_{i1}$;
    \item $|E(\Gamma)|$ squares: for each $(\ell, r)\in E(\Gamma)$ an oriented square $s_{\ell r}$ with the four distinct vertices, and edges (in order) $e_{\ell,0}, e_{1,r}, \overline e_{\ell,1}, \overline e_{0,r}, $ (where $\overline e$ denotes the edge $e$ with reversed orientation).
\end{itemize}

For $u \in V(\Gamma)$, let $\sigma_{u} = \begin{cases} 1 &\ \text{if }u\in S_{00} = L_0\cup R_0\\ -1 & \text{ if }u\notin S_{00} = L_0\cup R_0\end{cases}$. 
Since $\overline X_\Gamma$ is $2$-dimensional, we consider the chain complex
$C^{\overline\chi}_*(\overline X_\Gamma;\mathbb Q(t))$ of $\mathbb Q(t)$-vector spaces with bases
the cells of $\overline X_\Gamma$, non-zero only in dimensions $0,1,2$, with boundary maps
   $$\partial_1(e_{\ell, j}) = t^{\sigma_\ell}x_{1j} - x_{0j}  , \quad \partial_1(e_{i, r}) = t^{\sigma_r} x_{i1} - x_{i0},$$
   $$\partial_2(s_{\ell r}) = e_{\ell,0} + t^{\sigma_\ell} e_{1,r} - t^{\sigma_r} e_{\ell,1} - e_{0,r}.$$

\begin{lemma}\label{lem: twisted homology bipartite}
Suppose all states $S_{ij}$ in $\Gamma$ are non-empty. Then $$ \dim_{\mathbb Q(t)} H^1_{\overline\chi}(\overline X_{\Gamma}; \mathbb Q(t)) = \sum_{i,j\in \{0,1\}} \tilde b_0(\Gamma_{ij})$$
where $\tilde b_0(\Gamma_{ij})$ is the reduced $0$th Betti number of $\Gamma_{ij}$, i.e.\ the number of components of $\Gamma_{ij}$ minus $1$.
\end{lemma}

\begin{proof}
Note that $\dim_{\mathbb Q(t)}H^1_{\overline\chi}(\overline X_{\Gamma}; \mathbb Q(t)) = \dim_{\mathbb Q(t)} Z^1-\dim_{\mathbb Q(t)} B^1$ where $Z^1 = \ker \delta^1$, and $B^1 = \im \delta^0$. We first analyze the cocycles $Z^1$. 

Since $\overline X_{\Gamma}$ has $2|V|$ edges, the space of $1$-cochains $C^1_{\overline \chi}(\overline X_{\Gamma}; \mathbb Q(t))\simeq \mathbb Q(t)^{2|V(\Gamma)|}$, and the natural coordinates of $\phi\in C^1_{\overline \chi}(\overline X_{\Gamma}; \mathbb Q(t))$ are $a_{\ell,j} = \phi(e_{\ell,j}), b_{i,r} = \phi(e_{i,r})$ for $i,j\in\{0,1\}$, $\ell\in L$, $r\in R$.

A $1$-cochain $\phi$ is a cocycle if 
$$0 = \delta^1\phi (s_{\ell r}) = a_{\ell,0} + t^{\sigma_\ell} b_{1,r} - t^{\sigma_r} a_{\ell,1} - b_{0,r},$$
or equivalently $a_{\ell,0}  - t^{\sigma_r} a_{\ell,1} = b_{0,r} - t^{\sigma_\ell} b_{1,r}$. Consider a change of coordinates on $C^1_{\overline \chi}(\overline X_{\Gamma}; \mathbb Q(t))$:
$$\alpha_{\ell,0} = a_{\ell,0} - ta_{\ell,1}, \qquad \alpha_{\ell,1} = a_{\ell,0} - t^{-1}a_{\ell,1}, \qquad \beta_{0,r} = b_{0,r} - tb_{1,r}, \qquad \beta_{1,r} = b_{0,r} - t^{-1}b_{1,r}.$$
This is indeed a change of coordinates, since the matrix of the map $(a_{\ell,0}, a_{\ell,1})\mapsto (\alpha_{\ell,0}, \alpha_{\ell,1})$ is 
$$
\begin{bmatrix}
    1 & -t\\ 1 & -t^{-1} 
\end{bmatrix}
$$
which is invertible since its determinant is $t-t^{-1}\neq 0$(and similarly for the $b$-coordinates). In the new coordinates the condition for being a $1$-cocycle becomes 
\begin{align}\label{eq: cocycle condition}
    \alpha_{\ell, j} = \beta_{i,r} \text{ for } (\ell, r)\in E(\Gamma) \text{ where } \ell \in L_i \text{ and } r\in R_j.
\end{align}
Consider $\sqcup \mathcal S = S_{00}\sqcup S_{10}\sqcup S_{01} \sqcup S_{11}$ (which has cardinality $2|V(\Gamma)|$), and define a bijection between the coordinates $\alpha_{\ell, j}, \beta_{i,r}$ and vertices of $\sqcup \mathcal S$ as follows:
map $\alpha_{\ell,j}$ to $\ell \in S_{ij}$ if $\ell \in L_i$, and map $\beta_{i, r}$ to $r\in S_{ij}$ if $r\in R_{j}$. This yields a decomposition 
$$
C^1_{\overline \chi}(\overline X_{\Gamma}; \mathbb Q(t)) = \bigoplus_{i,j\in \{0,1\}} C^0(\Gamma_{ij}; \mathbb Q(t)).
$$
Under this identification, by~\eqref{eq: cocycle condition} the $1$-cocycles are precisely locally constant functions on $\Gamma_{ij}$ (meaning constant on each connected component). 
In particular 
$$
\dim Z^1 = \sum_{i,j\in \{0,1\}} b_0(\Gamma_{ij})
$$
where $b_0(\Gamma_{ij})$ denotes the number of connected components of $\Gamma_{ij}$.

We now analyze the $0$-cochains. Let $c\in C^0_{\overline\chi}(\overline X_{\Gamma}; \mathbb Q(t))$ and let $c_{ij} = c(x_{ij})$ be its natural coordinates. Then $c$ is a $0$-cocycle if 
$$c_{0,j} = t^{\sigma_\ell} c_{1,j} \text{ for all }\ell\in L, \text{ and } c_{i,0} = t^{\sigma_r} c_{i,1} \text{ for all }r\in R.$$
Since we are assuming each $S_{ij}\neq\emptyset$, either both $L_0, L_1$ are non-empty, or both $R_0, R_1$ are nonempty, so we get $c_{0,j} = t c_{1,j} = t^{-1} c_{1,j}$ or $c_{i,0} = t c_{i,1} = t^{-1} c_{i,1}$. In either case, we must have $c\equiv 0$. Therefore $\delta^0$ is injective, and $\dim_{\mathbb Q(t)} B^1 = \dim_{\mathbb Q(t)} C^0_{\overline\chi}(\overline X_{\Gamma}; \mathbb Q(t)) = 4$.
Thus $\dim_{\mathbb Q(t)}H^1_{\overline\chi}(\overline X_{\Gamma}; \mathbb Q(t)) = \dim Z^1-\dim B^1 = \sum_{i,j\in \{0,1\}} b_0(\Gamma_{ij}) - 4 = \sum_{i,j\in \{0,1\}} \tilde b_0(\Gamma_{ij})$, as required.
\end{proof}

\begin{theorem}\label{thm: iff for finite genration}
    Let $\Gamma$ be a bipartite graph with the vertex set partition $V(\Gamma) = L\sqcup R$. Let $\mathcal M = \{L, R\}$ be the bipartite system of moves, and let $\mathcal S$ be an orbit of non-empty states under $\mathcal M$. Let $\chi: W_{\Gamma}'\to \mathbb Z$ be the associated homomorphism.
 
    Then $\ker \chi$ is finitely generated if and only if $\mathcal S$ is a legal orbit. 
\end{theorem}

\begin{proof}
    The backward direction is exactly Theorem~\ref{thm: JNW}.
    By Remark~\ref{rem: 4vertex complex} the homomorphism $\chi:W_\Gamma'\to \mathbb Z$ extends to a homomorphism $\overline \chi: \pi_1\overline X_{\Gamma}\to \mathbb Z$. The image of $\overline \chi$ is $2\mathbb Z$, since every loop in $\overline X_\Gamma$ has even length, and $\overline\chi(\ell\ell')=\pm2$ for
$\ell\in L_0$, $\ell'\in L_1$ (or similarly for $r\in R_0$, $r'\in R_1$). 
 Since $[\pi_1 \overline X_{\Gamma}: W_{\Gamma'}]<\infty$, the finite generation of $\ker \chi$ and $\ker\overline \chi$ are equivalent. Suppose $\ker \chi$ is finitely generated.

By Proposition~\ref{prop: cocycle criterion} applied to the surjection $\overline\chi/2$,
with $t^2$ corresponding to $2\in 2\mathbb Z = \im(\overline \chi)$, we get
$H_1^{\overline\chi/2}(\overline X_\Gamma;\mathbb Q(t^2))=0$. With respect to suitably rescaled
lifts of the cells, $C^{\overline\chi/2}_*(\overline X_\Gamma;\mathbb Q(t^2))\otimes_{\mathbb Q(t^2)}\mathbb Q(t)$
is the complex $C^{\overline\chi}_*(\overline X_\Gamma;\mathbb Q(t))$. Since
$\mathbb Q(t)\supseteq\mathbb Q(t^2)$ is a field extension, we get
$$H_1^{\overline\chi}(\overline X_\Gamma;\mathbb Q(t)) = 0.$$

    By Lemma~\ref{lem: twisted homology bipartite}, that means $\sum_{i,j\in \{0,1\}} \tilde b_0(\Gamma_{ij}) = 0$. Thus $\tilde b_0(\Gamma_{ij}) = 0$ for each $i,j\in \{0,1\}$, i.e.\ each graph $\Gamma_{ij}$ is connected. By Lemma~\ref{lem:bipartite moves states} $\mathcal S$ is a legal orbit of states under $\mathcal M$.
\end{proof}

Theorem~\ref{thm: iff for finite genration} together with Corollary~\ref{cor: problematic vertex in Lambda} and Theorem~\ref{thm: no legal system operation} yield the following.
\begin{corollary} Let $\Gamma = \Lambda(n,k)$ for $k\geq 3$ and $n\geq 2k-1$. Then any homomorphism $\chi: W_{\Gamma}'\to \mathbb Z$ induced by the bipartite system of moves and an orbit of non-empty states has infinitely generated kernel. 

Similarly, if $n> 2k-1$ and $\Gamma$ is any graph obtained from $\Lambda(n,k)$ by a sequence of vertex link doubling, simultaneous vertex link multiplying, or edge quadrupling, then any homomorphism $\chi: W_{\Gamma}'\to \mathbb Z$ induced by the bipartite system of moves and an orbit of non-empty states has infinitely generated kernel. 
\end{corollary}

\begin{remark}\label{rem: more general maps}
    There exist other diagonal maps on $X_{\Gamma}$ (and more generally on $X_{\Gamma'}$ for a graph $\Gamma'$ whose right-angled Coxeter group $W_{\Gamma'}$ is commensurable with $W_{\Gamma}$) than the diagonal maps induced by bipartite system of moves. Even more generally, there also exist homomorphisms $W_{\Gamma}'\to \mathbb Z$ which are not induced by diagonal maps on $X_{\Gamma}$. We do not make claims about any such homomorphisms.
\end{remark}

We have not determined if $W_{\Lambda(n,k)}$ with $k\geq 3$, $n\geq 2k$ virtually algebraically fibers, and our strategy might be insufficient to resolve it.
\begin{question}
    Does there exist a finite index subgroup $H$ of $W_{\Lambda(n,k)}$ generated by reflections such that its defining graph as a right-angled Coxeter group admits a legal system with the bipartite system of moves? 
\end{question}
In view of Remark~\ref{rem: more general maps} the following questions also remain open.
\begin{question}
    Does $X_{\Lambda(n,k)}$ admit a diagonal map inducing an algebraic fibering? More generally, does the commutator subgroup of $W_{\Lambda(n,k)}$ algebraically fibers?
\end{question}

\subsection{Failure of Theorem~\ref{thm: iff for finite genration} for more general systems of moves}
The following example illustrates that the condition in Theorem~\ref{thm: JNW} is sufficient but not necessary for more general systems of moves, even in bipartite graphs. 

\begin{example}\label{example: non legal, finite generation}
    Let $\Gamma$ be the complete bipartite graph $K_{2,3}$ with $L = \{a_1, a_2\}$, $R = \{b_1, b_2, b_3\}$.
    Let $$M_{a_1} = L, \qquad M_{a_2}= \{a_2\}, \qquad M_{b_i} = R \text{ for }i=1,2,3.$$
    Consider the following orbit of states under $\mathcal M$:
    $$\mathcal S= \{\{b_1\}, \{b_2,b_3\}, \{a_1, b_1\}, \{a_1, b_2,b_3\}, \{a_2, b_1\}, \{a_2, b_2,b_3\}, \{a_1, a_2, b_1\}, \{a_1, a_2,b_2,b_3\}\}.$$ 
    This is not a legal orbit, as the subgraph spanned by $\{b_2, b_3\}$ is not connected. 
    We now consider a diagonal map $f:X_{\Gamma}\to S^1$ as described above associated with this system, and assume $S(x_0) = \{b_1\}$ for some $x_0\in X_{\Gamma}^0$. The induced homomorphism $\chi = f_*:W_{\Gamma}'\to \mathbb Z$ is non-zero by Lemma~\ref{lem: nontrivial homomorphism}.
    Since $\Gamma$ is a join of $\{a_1, a_2\}$ and $\{b_1, b_2, b_3\}$, $W_{\Gamma} = W_{\{a_1, a_2\}}\times W_{\{b_1, b_2, b_3\}}$, and
    $$W_{\Gamma}' = W_{\{a_1, a_2\}}'\times W_{\{b_1, b_2, b_3\}}' = \mathbb Z \times F$$ where $F$ is a free group of finite rank. The element $(a_1a_2)^2$ is a  generator of $W_{\{a_1, a_2\}}' = \mathbb Z$, and it is represented by a path 
    $x_0 \xrightarrow{a_1}x_1 \xrightarrow{a_2} x_2 \xrightarrow{a_1} x_3 \xrightarrow{a_2} x_0$ in $X_{\Gamma}$. Since $a_1\notin S(x_0) = \{b_1\}$, $a_2\in S(x_1) = S(x_0)\cdot M_{a_1} = \{a_1, a_2, b_1\}$, $a_1\in S(x_2) = S(x_1)\cdot M_{a_2} = \{a_1, b_1\}$, and $a_2 \in S(x_3) = S(x_2)\cdot M_{a_1} = \{a_2, b_1\}$, we have 
    $$\chi((a_1a_2)^2) = -1 + 1 + 1 + 1 = 2.$$
    By Lemma~\ref{lem: kernels from Z times H} (below), $\ker \chi$ projects isomorphically onto a finite index subgroup of $W_{\{b_1, b_2, b_3\}}'\simeq F$, and therefore it is finitely generated. 
\end{example}

\begin{lemma}\label{lem: kernels from Z times H}
    Let $G = \langle t\rangle \times H$ where $\langle t\rangle \simeq \mathbb Z$, and let $\chi:G\to \mathbb Z$. If $\chi(t)\neq 0$, then $\ker\chi$ projects isomorphically onto a finite index subgroup in $H$.
\end{lemma}
\begin{proof}
    Let $n = \chi(t)\in \mathbb Z$. View $H$ as a subgroup of $G$ via ${e}\times H$.
    Consider a projection map $\pi:\ker \chi \to H$. The map $\pi$ is injective, since $\chi(t^k, 1) = kn\neq 0$ when $k\neq 0$. We need to show that $[H:\pi(\ker \chi)]<\infty$. We claim that $\im(\pi) = \chi_{|H}^{-1}(n\mathbb Z)$. Indeed, if $(t^k, h)\in \ker \chi$, then $\chi_{|H}(h)=\chi (1, h) = \chi(t^{-k},1) = -nk\in n\mathbb Z$. Conversely, if $\chi_{|H}(h)= nk$, then $\chi(t^{-k}, h) = 0$. Thus $\chi_{|H}: H\to \mathbb Z$ induces an injective map $H/\chi_{|H}^{-1}(n\mathbb Z)\to \mathbb Z/n\mathbb Z$, so $[H:\pi(\ker \chi)]\leq |n|$. 
\end{proof}

\bibliographystyle{alpha}
\bibliography{references}

\end{document}